\documentclass{elsarticle}

\usepackage{hyperref}

\journal{arXiv}

\usepackage{textcomp}
\usepackage[utf8]{inputenc}
\usepackage[english]{babel}
\usepackage{amsmath}
\usepackage{amssymb}%

\usepackage[shortlabels]{enumitem}

\usepackage{amsthm}%
\usepackage{graphicx}
\usepackage{multirow}
\usepackage{xcolor}

\newtheorem{theorem}{Theorem}
\newtheorem{lemma}{Lemma}
\newtheorem{problem}{Problem}
\newtheorem{remark}{Remark}
\newtheorem{notation}{Notation}
\newtheorem{assumption}{Assumption}

\begin{document}

\begin{frontmatter}

\title{Robust multigrid solvers for the biharmonic problem in isogeometric analysis}

\author[mymainaddress]{Jarle Sogn\corref{mycorrespondingauthor}}
\author[mysecondaryaddress]{Stefan Takacs}
\cortext[mycorrespondingauthor]{Corresponding author}

\address[mymainaddress]{Institute of Computational Mathematics, 
Johannes Kepler University Linz,\\ 
Altenberger Str. 69, 4040 Linz, Austria}

\address[mysecondaryaddress]{Johann Radon Institute for Computational and Applied Mathematics (RICAM),\\
Austrian Academy of Sciences,
Altenberger Str. 69, 4040 Linz, Austria}

\begin{abstract}
In this paper, we develop multigrid solvers for the biharmonic problem in the 
framework of isogeometric analysis (IgA). In this framework, one typically sets up 
B-splines on the unit square or cube and transforms them to the domain of interest 
by a global smooth geometry function.
With this approach, it is feasible to set up $H^2$-conforming 
discretizations. We propose two multigrid methods for such a 
discretization, one based on Gauss Seidel smoothing and one based on mass 
smoothing. We prove that both are robust in the grid size, the latter is 
also robust in the spline degree. Numerical experiments illustrate the 
convergence theory and indicate the efficiency of the
proposed multigrid approaches, particularly of a hybrid approach combining
both smoothers.
\end{abstract}


\end{frontmatter}

\section{Introduction}

Isogeometric analysis (IgA) was introduced around a decade ago 
as a new paradigm to the discretization of 
partial differential equations (PDEs) and has gained increasing attention
(cf.~\cite{hughes2005isogeometric} for the original paper and
\cite{da2014mathematical} for a survey paper). The 
idea of IgA -- from the technical point of view -- is to use B-spline spaces or 
similar spaces, like NURBS spaces, to discretize the problem.

In contrast to standard $C^0$-smooth high-order finite elements, the 
introduction of discretizations  with higher smoothness on general computational 
domains is not straight forward. In IgA, splines are first set up on the 
unit square or the unit cube, which is usually called the parameter domain. 
Then, a global smooth geometry transformation mapping from the parameter domain to the physical domain, i.e., the domain of interest, is used to define the ansatz functions on the physical domain.

Such an approach allows to construct arbitrarily smooth ansatz functions. 
So, we easily obtain $H^2$-conforming discretizations which can be used as
conforming discretizations of the biharmonic problem, which is 
for example of interest in plate theory (cf.~\cite{ciarlet2002finite}), Stokes streamline 
equations (cf.~\cite{girault2012finite}), or Schur 
complement preconditioners (cf.~\cite{KAMmagneOC, SZschur}). For the latter, also the three 
dimensional version of the biharmonic problem is of interest. Such $H^2$-conforming
discretizations are hard to realize in a standard finite element scheme. One option is
the Bogner-Fox-Schmit element, which requires a rectangular mesh, another option is the
Argyris elements for triangular meshes.
For such $H^2$-conforming elements, besides various kinds of other preconditioners
(cf.~\cite{doi:10.1137/15M1014887} and references therein), also
multigrid solvers have been proposed (cf.~\cite{zhang1989optimal}).
As alternative, multigrid solvers for various kinds of mixed or non-conforming formulations
have been developed
(cf.~\cite{doi:10.1137/0726062,Zhang:Xu, hanisch1993multigrid} and references therein).

In this paper, we develop iterative solvers for conforming Galerkin 
discretizations of the biharmonic problem in an isogeometric setting. 
Multigrid methods are known to solve linear systems arising from 
the discretization of partial differential equations with optimal 
complexity, i.e., their computational complexity grows typically only linearly with the 
number of unknowns. In an isogeometric setting, multigrid and multilevel methods
have been discussed within the last years 
(cf.~\cite{BuffaEtAl:2013,gahalaut2013multigrid,HofreitherZulehnerFull,HTZ:2016,
hofreither2016robust}). It was observed that multigrid methods based on standard 
smoothers, like the Gauss Seidel smoother, show robustness in the grid size 
within the isogeometric setting, their convergence rates however deteriorate 
significantly if the spline degree is increased. This motivated the recent 
publications~\cite{HTZ:2016,hofreither2016robust}. In the latter, a subspace 
corrected mass smoother was introduced, based on the approximation error 
estimates and inverse inequalities from~\cite{takacs2016approximation}.

The present paper is a continuation 
of~\cite{takacs2016approximation} and~\cite{hofreither2016robust}. We  
propose two multigrid methods for the linear system resulting from the 
discretization of the biharmonic problem, one based on Gauss Seidel smoothing and
one based on a subspace corrected mass smoother.
We prove that both are robust in the grid size, the latter is also robust in the spline degree. 
For this purpose, non-trivial extensions to both previous
papers are required. \cite{takacs2016approximation} covers the
approximation with functions whose odd derivatives vanish on the boundary;
an extension to functions whose even derivatives (including the function value itself) vanish
on the boundary might be straight-forward, however, for the first biharmonic problem we need a combination of both. A straight-forward extension of~\cite{hofreither2016robust} would require full $H^4$ regularity, which cannot even be assumed on the unit square (cf.~\cite{blum1980boundary}).
So, we only require partial regularity (Assumption~\ref{ass:reg}) and derive the convergence
results using Hilbert space interpolation.

We give numerical experiments both for domains described by trivial and  non-trivial geometry transformation in two and three dimensions. We observe that the subspace corrected mass smoother outperforms the Gauss Seidel smoother for significant large spline degrees. The negative effects of the geometry transformation to the subspace corrected mass smoother, which have also been observed for the Poisson problem, are amplified in case of the biharmonic problem. 
Approaches to master these effects are of particular interest for the biharmonic problem.
We propose a hybrid smoother which combines the strengths of both proposed smoothers and 
works well in our numerical experiments (cf. Section~\ref{sec:hybrid}).

The remainder of the paper is organized as follows. We introduce the model problem and its discretization in Section~\ref{sec:preliminaries}. Then, in Section~\ref{sec:approx:inverse}, we develop the required approximation error estimates. In Section~\ref{sec:splitting}, we set up a stable splitting of the spline spaces. In Section~\ref{sec:MGconstruction}, we introduce the multigrid algorithms and prove their convergence. Finally, in Section~\ref{sec:num}, we give results from the numerical experiments and draw conclusions.

\section{Preliminaries} 
\label{sec:preliminaries}
\subsection{Model problem}

In this paper, we consider the  first biharmonic problem as model problem, which reads as
follows. For a given domain $\Omega \subset \mathbb R^d$ with piecewise $C^2$-smooth Lipschitz
boundary $\Gamma = \partial \Omega$ and a given source function $f$, find the unknown function $u$ such that
\begin{align}
\Delta^2 u &= f \quad \text{in}\quad\Omega,\nonumber\\
         u &= 0 \quad \text{on}\quad\Gamma,\label{eq:mp1}\\
\nabla u \cdot \mathbf{n} &= 0 \quad \text{on}\quad\Gamma,\nonumber
\end{align} 
where $\mathbf{n}$ is the outer normal vector; for simplicity, we restrict ourselves
to homogenous boundary conditions. Our proposed solver can be extended
to other boundary conditions, namely
to the second and the third biharmonic problem, cf. Remarks~\ref{rem:second} and~\ref{rem:third}.

Following the principle of IgA, we assume that the computational domain $\Omega$ is represented by a
bijective geometry transformation
\begin{equation}\label{eq:geotrans}
	\mathbf{G}: \widehat{\Omega} \rightarrow \Omega
\end{equation}
mapping from the parameter domain $\widehat{\Omega}:=(0,1)^d$ to the physical domain $\Omega$.

The variational formulation of model problem~\eqref{eq:mp1} is as follows. 
\begin{problem}
\label{problem:bi1}
Given $f\in L^2\left(\Omega \right)$, find $u\in V:= H^{2}_{0}\left(\Omega\right)$ such that
\begin{equation}\label{eq:mcB}
\underbrace{ \left(\Delta u, \Delta v\right)_{L^2\left(\Omega\right)} }_{\displaystyle \left(u, v\right)_{\mathcal{B}(\Omega)}:=} =  \left(f, v\right)_{L^2\left(\Omega\right)}  \quad \forall\, v\in V.
\end{equation}
\end{problem}
Here and in what follows, $L^2$ and $H^r$ denote the standard Lebesgue and Sobolev spaces
with standard inner products $(\cdot,\cdot)_{L^2}$, $(\cdot,\cdot)_{H^r}$, norms 
$\|\cdot\|_{L^2}$, $\|\cdot\|_{H^r}$ and seminorms $|\cdot|_{H^r} = (\cdot,\cdot)_{H^r}^{1/2}$.
$H^{2}_{0}\left(\Omega\right)$ is the standard subspace of $H^2$, containing the functions where the values and the derivatives
vanish on the boundary, i.e.,
\[
    H^{2}_{0}\left(\Omega\right) =\left\lbrace v \in H^{2}\left(\Omega\right) \vert \, v = \nabla v\cdot n = 0 \text{ on }\Gamma\right\rbrace.
\]

Note that the inner products
$(\cdot,\cdot)_{H^2(\Omega)}$ and $\left(\cdot, \cdot\right)_{\mathcal{B}(\Omega)}$
coincide on $H^{2}_{0}\left(\Omega\right)$ (cf.~\cite{grisvard2011elliptic}), i.e., 
\begin{equation}\label{eq:mcB:equiv}
		\left(u, v\right)_{\mathcal{B}(\Omega)} = (u,v)_{H^2(\Omega)}\qquad \forall u,v \in H^2_0(\Omega).
\end{equation}

Let $\left(\cdot, \cdot\right)_{\bar{\mathcal{B}}(\Omega)}$ be the inner product obtained by removing
the cross terms from the inner product $\left(\cdot, \cdot\right)_{\mathcal{B}(\Omega)}$, i.e.,
\begin{equation}\label{eq:mcBbar}
\left(u, v\right)_{\bar{\mathcal{B}}(\Omega)} := \sum^d_{k=1}\left(\partial_{x_kx_k} u , \partial_{x_kx_k} u \right)_{L^2\left(\Omega\right)}.
\end{equation}
Here and in what follows, $\partial_{x} := \frac{\partial}{\partial x}$ and $\partial_{xy} := \partial_{x} \partial_{y}$
and $\partial_{x}^r := \frac{\partial^r}{\partial x^r}$ denote partial derivatives. 
\begin{lemma}
\label{lemma:BdecompEquiv2D}
The inner products defined in~\eqref{eq:mcB} and~\eqref{eq:mcBbar}
are spectrally equivalent, i.e.,
\begin{equation*}
\left(u, u\right)_{\bar{\mathcal{B}}(\Omega)} \leq \left(u, u\right)_{\mathcal{B}(\Omega)} \leq d\left(u, u\right)_{\bar{\mathcal{B}}(\Omega)}\quad \forall\, u\in H^2_0(\Omega).
\end{equation*} 
\end{lemma}
\begin{proof}
	Using the Cauchy-Schwarz inequality and $ab \le \tfrac12( a^2+ b^2)$, we obtain
	\begin{align*}
	\|u\|_{\mathcal{B}(\Omega)}^2 &= \sum_{k=1}^d\sum_{l=1}^d \left(\partial_{x_kx_k} u,\partial_{x_lx_l} u\right)_{L^2\left(\Omega\right)}
		\\&\le \frac{1}{2}  \sum_{k=1}^d\sum_{l=1}^d \left( \left\|\partial_{x_kx_k} u\right\|_{L^2\left(\Omega\right)}^2 + \left\|\partial_{x_lx_l} u\right\|_{L^2\left(\Omega\right)}^2\right)
		= d \|u\|_{\bar{\mathcal{B}}(\Omega)}^2,
	\end{align*}
	which shows one direction. Using the boundary conditions and~\eqref{eq:mcB:equiv}, we obtain
	\begin{align*}
	\|u\|_{\mathcal{B}(\Omega)}^2
		& = \|u\|_{H^2(\Omega)}^2
		 = \|u\|_{\bar{\mathcal{B}}(\Omega)}^2 + 
		\sum_{k=1}^d\sum_{l\in\{1,\ldots,d\}\backslash\{k\}} \underbrace{\left(\partial_{x_kx_l} u,\partial_{x_kx_l}u\right)_{L^2\left(\Omega\right)} }_{\displaystyle \ge 0},
	\end{align*}
	which shows the other direction.
\end{proof}

\subsection{Spline space}

We consider standard tensor product B-spines with maximum continuity (see, e.g., \cite{de1978practical}).
Let the interval $(0,1)$ be subdivided into $m\in \mathbb{N}$ elements of length $h=1/m$. 
The space of splines of degree $p\in \mathbb{N}:=\{1,2,3,\ldots\}$ with maximum continuity is defined by
\begin{equation*}
S_{p,h}(0,1):= \left\lbrace u\in C^{p-1}\left(0,1\right): u\vert_{\left(\left(i-1\right)h,ih\right)}\in \mathcal{P}^p \quad \forall\, j=1,\ldots m\right\rbrace,
\end{equation*} 
where $C^{p-1}\left(0,1\right)$ is the space of all $p-1$ times continuously differentiable functions on $\left(0,1\right)$ and $\mathcal{P}^p$ is the space of all polynomials with degree at most $p$. We use the standard $B$-splines with open knot vector as basis for $S_{p,h}(0,1)$. The dimension of $S_{p,h}(0,1)$ is $n:=\dim{S_{p,h}(0,1)}= m+p$. We will from time to time omit the subscripts $p$ and $h$ of a spline space $S_{p,h}(0,1)$ and write $S(0,1)$ or just $S$. For higher dimensions $d>1$, we use the tensor product splines
\begin{equation*}
           S_{p,h}(\widehat{\Omega})= S_{p,h}(0,1)\otimes \ldots \otimes S_{p,h}(0,1),
\end{equation*}
defined over $\widehat{\Omega} = \left(0,1\right)^d$. For notational convenience, we assume that all of those univariate spline spaces $S_{p,h}$
have the same spline degree $p$ and the same number of elements $m$, however, this in not necessary and the results in this paper can easily
be generalized  to the case with different $p$ and $m$.   

Based on the spline space on the parameter space, we define the spline space on the
physical space using the standard pull-back principle as
\[
	S_{p,h}(\Omega) = \{ u \; :\; u \circ \mathbf{G} \in S_{p,h}(\widehat{\Omega})\},
\]
where $\mathbf{G}$ is the geometry transformation~\eqref{eq:geotrans}. We assume that the geometry transformation
is sufficiently smooth such that the following estimate holds.
\begin{assumption} 
\label{ass:GeoEqui}
Assume that there exist constants $\underline{\alpha}>0$ and  $\overline{\alpha}$ such that
\begin{align*}
\underline{\alpha}\,\|u \|_{H^q(\Omega)} \leq \|u\circ \mathbf{G} \|_{H^q(\widehat{\Omega})} \leq \overline{\alpha}\, \|u \|_{H^q(\Omega)} \quad \forall\, u \in H^q(\Omega),\, q\in\{2,3\}.
\end{align*}
\end{assumption}

We discretize the Problem~\ref{problem:bi1} using the Galerkin principle as follows. 
\begin{problem}
\label{problem:bi1:disc}
Given $f\in L^2(\Omega )$, find $u\in V_h:= S_{p,h}^0(\Omega):= H^{2}_{0}\left(\Omega\right) \cap S_{p,h}(\Omega)$ such that
\begin{equation}
\label{eq:Hspaces1}
\left( u,  v\right)_{\mathcal{B}(\Omega)} =  \left(f, v\right)_{L^2\left(\Omega\right)}  \quad \forall\, v\in V_h.
\end{equation}
\end{problem}

By fixing a basis for the space $S_{p,h}^0(\Omega)$, we can rewrite the Problem~\ref{problem:bi1:disc} in matrix-vector notation as
\begin{equation}\label{eq:lin:sys}
        \mathcal B_h u_h = f_h,
\end{equation}
where $\mathcal B_h$ is a standard stiffness matrix, $u_h$ is the representation of the corresponding function
$u$ with respect to the chosen basis and the vector $f_h$ is obtained by testing the right hand side functional
$(f,\cdot)_{L^2(\Omega)}$ with the basis functions.

For convenience, we use the following notation.
\begin{notation}\label{not:2:1}
    Throughout this paper, $c$ is a generic positive constant independent of $h$ and $p$, but may depend on $d$ and $\mathbf{G}$.
\end{notation}

\subsection{Regularity}

In the following sections, we use Aubin-Nitsche duality arguments for showing the desired error estimates. This requires
that the following assumption holds.
\begin{assumption}\label{ass:reg}
        For a given $f\in H^{-1}(\Omega)$, the solution $u\in H^{2}_{0}\left(\Omega\right)$ of the first biharmonic problem~\eqref{eq:mp1} satisfies
        \[
                u\in H^3(\Omega)\qquad \mbox{and}\qquad \|u\|_{H^3(\Omega)} \le c \|f\|_{H^{-1}(\Omega)}.
        \]
\end{assumption}
Such a result is satisfied for convex polygonal domains (cf. \cite{blum1980boundary}). 
It is worth noting that this implies that the result also holds for the parameter domain $\widehat{\Omega}=(0,1)^2$.

As we only rely  on a partial regularity result, we use Hilbert space interpolation (cf.~\cite{lofstrom1976interpolation,Adams:Fournier})
to derive our estimates.
defined, e.g., with the K-method, is a Hilbert space with norm
$
    \|\cdot\|_{[A_1,A_2]_{\theta}}.
$
Applied to Sobolev spaces 
$H^{m}(\Omega)$ and $H^{n}(\Omega)$, we obtain 
\begin{equation}\label{eq:sobolev}
      \|\cdot\|_{[H^{m}(\Omega),H^{n}(\Omega)]_{\theta}}^2 = \|\cdot\|_{H^{(1-\theta)m+\theta n}(\Omega)}^2,
\end{equation}
see~\cite[Theorem~6.4.5]{lofstrom1976interpolation},
applied to scaled Hilbert spaces $A_1$ and $\gamma A_2$  with a scaling parameter $\gamma>0$, we obtain
\begin{equation}\label{eq:scaling}
      \|\cdot\|_{[ A_1, \gamma A_2 ]_{\theta}}^2 = \gamma^{\theta} \|\cdot\|_{[A_1,A_2]_{\theta}}^2
\end{equation}
and applied to the intersections of two Hilbert spaces $A_1 \cap A_2$ with norm ${\|\cdot\|_{A_1\cap A_2}^2:= \|\cdot\|_{A_1}^2+\|\cdot\|_{A_2}^2}$, we obtain
\begin{equation}\label{eq:hilbsum}
      \|\cdot\|_{[A_1, A_1 \cap A_2 ]_{\theta}}^2 \le c \|\cdot\|_{ A_1 \cap [A_1,A_2]_{\theta} }^2,
\end{equation}
see~\cite[Lemma~6.1]{TZ:2013}, and applied to dual norms, we obtain
\begin{equation}\label{eq:hilbdual}
      \|\cdot\|_{([A_1,A_2]_{\theta})'}^2 = \|\cdot\|_{ [A_1',A_2']_{\theta} }^2,
\end{equation}
see~\cite[Theorem~3.7.1]{lofstrom1976interpolation}.
As the interpolation defined by the K-method is an \emph{exact interpolation function},
see~\cite[Theorem~3.1.2]{lofstrom1976interpolation}, we know that any
bounded operator $\Psi$, which maps from a Hilbert space $A_1$ to a Hilbert space $B_1$ and
from a Hilbert space $A_2$ to a Hilbert space $B_2$, 
maps also from $[A_1,A_2]_{\theta}$ to $[B_1,B_2]_{\theta}$ and satisfies 
\begin{equation}\label{eq:h1}
		\|\Psi a\|_{[B_1,B_2]_{\theta}} \le c M_1^{1-\theta} M_2^{\theta} \|a\|_{[A_1,A_2]_{\theta}}
	\qquad \mbox{with}\quad M_i := \sup_{a_i\in A_i} \frac{\|\Psi a_i\|_{B_i}}{\|a_i\|_{A_i}}
\end{equation}
for all $\theta\in(0,1)$,
where $c$ only depends on $\theta$.

\section{Approximation error estimates}
\label{sec:approx:inverse}
One vital component needed to prove multigrid convergence is an approximation error estimate. Approximation error estimates between the spaces $L^2\left(\Omega\right)$ and $H^{1}\left(\Omega\right)$ are given in \cite{takacs2016approximation,floater2017optimal} and used in \cite{HTZ:2016,hofreither2016robust} to prove convergence for a multigrid solver for the Poisson problem. For the biharmonic problem we need similar estimates for $H^{2}\left(\Omega\right)$.

\subsection{Approximation error estimates for the periodic case}

We start the analysis for the periodic case. We define for each $q\in \mathbb N$ the periodic Sobolev space
\[
H^{q}_{per} (-1,1):= \left\lbrace u \in H^q(-1,1)  \,:\, u^{\left(l\right)}\left(-1\right)=u^{\left(l\right)}\left( 1\right), \quad \forall\, l\in\mathbb N_0 \mbox{ with } l<q \right\rbrace
\]
and for each $p\in \mathbb N$ the periodic spline space
\[
S_{p,h}^{per}(-1,1) := \left\lbrace u \in S(-1,1)  \,:\, u^{\left(l\right)}\left(-1\right)=u^{\left( l\right)}\left( 1\right)\quad \forall\, l\in\mathbb N_0 \mbox{ with } l<p \right\rbrace.
\]
Let
$T_{p,h}^{q,per}$ be the $H^{q,\circ}$-orthogonal projection into $S_{p,h}^{per}(-1,1)$, where 
the underlying scalar product $\left(\cdot, \cdot\right)_{H^{q,\circ}(-1,1)}$ is given by
\begin{align*}
\left(u, v\right)_{H^{r,\circ}(-1,1)} :=
\left\{
	\begin{array}{ll}
		\left(u, v\right)_{H^q(-1,1)} + \frac12 \int^1_{-1} u\,\text{d}x \int^1_{-1} v \,\text{d}x &\;\text{ for }\; q> 0, \\
		\left(u, v\right)_{L^2(-1,1)} &\;\text{ for }\; q=0,
	\end{array}
\right.
\end{align*}
where $\frac12\int^1_{-1} u\,\text{d}x \int^1_{-1} v \,\text{d}x $ is added to enforce uniqueness.

\begin{theorem}\label{thrm:approx1d:prelim}
	Let $p\in \mathbb{N}_0$, $q\in \mathbb{N}_0$ with $p \ge q$ and $hp<1$. Then,
	\[
		| (I-T_{p,h}^{q,per})u  |_{H^{q}(-1,1)} \le \sqrt{2} h | u |_{H^{q+1}(-1,1)} \quad \forall u\in H^{q+1}_{per}(-1,1).
	\]
\end{theorem}
\begin{proof}
	We use induction with respect to $q$.
	
	\emph{Proof for $q=0$.}
	\cite[Lemma~4.1]{takacs2016approximation} gives an approximation error estimate for the $H^{1,\circ}$-orthogonal projection
	of $u$ into $S_{p,h}^{per}$ for $p\ge1$. Because $T_{p,h}^{0,per}$ minimizes the $L^2$-norm, we obtain
    \[
        \| (I-T_{p,h}^{0,per})u  \|_{L^2} 
        \le \| (I-T_{p,h}^{1,per})u  \|_{L^2} \le \sqrt{2} h | u |_{H^1} \quad \forall\, u \in H^1_{per}(-1,1),
    \]
    i.e., the desired result.
    For $p=0$, we observe that there are no periodicity conditions for the space $S_{p,h}^{per}$. The desired result
    on approximation by piecewise constants is standard and can be found, e.g, in \cite[Theorem~6.1]{schumaker2007spline}.

    \emph{Proof for $q>0$.}
    We already know that the induction hypothesis holds true for $q-1$, i.e., we have
    \begin{equation}\label{eq:approx1d:1}
        | u - T_{p-1,h}^{q-1,per} u |_{H^{q-1}(-1,1)} \le \sqrt{2} h | u |_{H^{q}(-1,1)} \quad \forall\, u \in H^{q}_{per}(-1,1).
    \end{equation}

    As a next step we show that for all $u\in H^{q+1}_{per}(-1,1)$, there is a $u_h \in S_{p,h}^{per}$ such that 
	\begin{equation}\label{eq:approx1d:2}
		|u-u_h|_{H^{q}(-1,1)} \le \sqrt{2} h | u |_{H^{q+1}(-1,1)}.
	\end{equation}
	
	By plugging $u'$ into~\eqref{eq:approx1d:1}, we immediately obtain
    \[
        | u' - T_{p-1,h}^{q-1,per} u' |_{H^{q-1}(-1,1)} \le \sqrt{2} h  | u |_{H^{q+1}(-1,1)} \quad \forall\, u \in H^{q+1}_{per}(-1,1).
    \]
	Let $v_h:=T_{p-1,h}^{q-1,per} u'$ and define $u_h(x) := \int_{-1}^x  v_h(\xi) \text d\xi+\gamma$, where $\gamma\in \mathbb{R}$ such that
	$\int_{-1}^1u_h(x)\text dx=0$. For this choice, we obtain the desired
	estimate~\eqref{eq:approx1d:2}. It remains to show $u_h \in S_{p,h}^{per}$. As we have
	$v_h\in  S_{p-1,h}^{per}$, we obtain that $u_h$ is a spline of degree $p$. The continuity
	estimates
	\[
			u_h^{(l)} (-1) = u_h^{(l)} (1)  \quad\mbox{for } l=1,\ldots,p-1
	\]
	follow directly from $v_h^{(l)} (-1) = v_h^{(l)} (1)  \quad\mbox{for } l=0,\ldots,p-2$. So, it remains
	to show $u_h(-1) = u_h(1)$. Note that, as $u$ is periodic, we have
    \begin{align*}
		u_h(-1) - u_h(1)  &=  (u(1)-u(-1)) - (u_h(1) - u_h(-1)) = \int_{-1}^1 u'(x)-u_h'(x) \text dx  \\&= \int_{-1}^1 v(x)-v_h(x)  \text dx = ((I-T_{p-1,h}^{q-1,per})v,1)_{L^2(-1,1)}.
    \end{align*}
    Note that $(\cdot,1)_{H^{r,\circ}(-1,1)}= (\cdot,1)_{L^2(-1,1)}$ for any $r$, so we obtain
    \[
            u_h(-1) - u_h(1) = ((I-T_{p-1,h}^{q-1,per})v,1)_{H^{q-1,\circ}(-1,1)}
    \]
    and finally, as $1\in S_{p-1,h}^{per}$, Galerkin orthogonality shows that this term is $0$.  
    So, we have shown $u_h \in S_{p,h}^{per}$ and~\eqref{eq:approx1d:2}. As the projector $T_{p,h}^{q,per}$ minimizes
    the $H^q$-seminorm, we obtain
    \[
        | (I-T_{p,h}^{q,per})u  |_{H^{q}(-1,1)} \le |u-u_h|_{H^{q}(-1,1)} \le \sqrt{2} h | u |_{H^{q+1}(-1,1)},
    \]
    i.e., the desired result.
\end{proof}

\subsection{Approximation error estimates for the univariate case}

Now, we derive approximation error estimates for univariate splines that satisfy the desired boundary
conditions. First, we consider the approximation of functions in the
Sobolev space of functions with vanishing even derivatives (and function values) on the boundary, given by
\[
H^{q}_{D} (0,1):= \left\lbrace u \in H^q(0,1)  \,:\, u^{\left(2l\right)}\left(0\right)=u^{\left(2l\right)}\left( 1\right) = 0, \quad \forall\, l\in\mathbb{N}_0\text{ with } 2l< q \right\rbrace,
\]
by functions in a corresponding spline space, given by
\[
S^{D,0}(0,1) :=\left\lbrace u\in S(0,1) : u^{\left(2l\right)}\left(0\right) = u^{\left(2l\right)}\left(1\right) = 0\quad \forall\, l\in\mathbb{N}_0\text{ with } 2l< p\right\rbrace.
\]
Now, we define $\Pi^{D,0}$ to be the $H^2$-orthogonal projection from $H^2_D(0,1)$ into $S^{D,0}(0,1)$. This projector
satisfies the following error estimate.

\begin{theorem}
\label{theorem:appNT2}
Let  $p\in \mathbb{N}$  with $p \ge 3$ and $hp<1$. Then,
\begin{equation*}
    | (I-\Pi^{D,0}) u |_{H^2(0,1)} \le 2 h^2 | u |_{H^4(0,1)} \quad \forall u\in H^4_D(0,1)
\end{equation*}
\end{theorem}
\begin{proof}
    Assume $u\in H^4_D(0,1)$ to be arbitrary but fixed. Define $w$ on $(-1,1)$ by
    \[
            w(x) := \mbox{sign} (x)\,u( |x| )
    \]
    and observe that $w \in H^4_{per}(-1,1)$. Using Theorem~\ref{thrm:approx1d:prelim}, we obtain
    \begin{align*}
        | (I-T_{p,h}^{2,per}) w |_{H^2(-1,1)} &= | (I-T_{p,h}^{2,per})(I-T_{p,h}^{3,per}) w |_{H^2(-1,1)} \\
&\leq \sqrt{2}h| (I-T_{p,h}^{3,per}) w |_{H^3(-1,1)}\le {2} h^2 | w |_{H^4(-1,1)}.
    \end{align*}

    First observe that $|w|_{H^4(-1,1)} = \sqrt{2} |u|_{H^4(0,1)}$. Define $w_h:= T_{p,h}^{2,per} w$ and
    observe that we obtain $ w_h(x)=-w_h(-x)$ using a standard symmetry argument. This implies that $u_h$,
    the restriction of $w_h$ to  $(0,1)$, satisfies $u_h\in S^{D,0}$. Moreover, we have
   $|w-w_h|_{H^2(-1,1)} = \sqrt{2} |u-u_h|_{H^2(0,1)}$ and, as a consequence,
    \[
        | u-u_h |_{H^2(0,1)} \le 2 h^2 | u |_{H^4(0,1)}.
    \]
    As the projector $\Pi^{D,0}$ minimizes the $H^2$-seminorm, the desired result follows.
\end{proof}

Now, we consider the boundary conditions of interest for the first biharmonic problem. Here, the continuous space is $H^2_0(0,1)$
and the discretized space is $S^0(0,1)$, given by
\[
	S^0(0,1) := \left\lbrace u \in S(0,1)  \,:\, u(0)=u'(0)=u(1)=u'(1)=0\right\rbrace = S(0,1)\cap H^2_0(0,1).
\]
Now, we define $\Pi^{0}$ to be the $H^2$-orthogonal projection from $H^2_0(0,1)$ into $S^{0}(0,1)$. This projector
satisfies the following error estimate.

\begin{theorem}\label{thrm:appr:c:2}
    Let  $p\in \mathbb{N}$  with $p \ge 3$ and $hp<1$. Then,
    \[
        | (I-\Pi^0) u |_{H^2(0,1)} \le 2 h^2 |u|_{H^4(0,1)}\quad \forall\, u\in H^4(0,1)\cap H^2_0(0,1).
    \]
\end{theorem}
\begin{proof}
	First, define $S^*(0,1) :=S\cap H^1_0(0,1)$ and $\Pi^*:H^2_D(0,1)\rightarrow S^*$
	to be the $H^2$-orthogonal projector into the corresponding space.
Since $S^{D,0}\subset S^*$, Theorem~\ref{theorem:appNT2} directly implies
\begin{align*}
    |\left(I-\Pi^*\right)w|_{H^2(0,1)}  \leq 2 h^2 \vert w\vert_{H^4(0,1)}
		\quad \forall\, w\in H^{4}_D(0,1).
\end{align*}
Now let $u\in H^{4}(0,1) \cap H^1_0(0,1)$ be arbitrary but fixed. Observe that for
\[
	w(x):=u(x) + \underbrace{\frac{1}{6}(x^3-3x^2+2x)}_{\displaystyle \phi_1(x):=} u''(0) - 
		\underbrace{\frac{1}{6} (x^3-x)}_{\displaystyle \phi_2(x):=} u''(1) ,
\]
we obtain $w\in H^4_D(0,1)$. Note that $\phi_1,\phi_2\in S^*$ and  $|\phi_1|_{H^4(0,1)}=|\phi_2|_{H^4(0,1)}=0$.
So,
\begin{align*}
	|\left(I-\Pi^*\right)u|_{H^2(0,1)}  &= \inf_{u^*\in S^*}|u-u^*|_{H^2(0,1)} \\
		& = \inf_{w^*\in S^*}|w-w^*|_{H^2(0,1)} \leq 2 h^2 \vert w\vert_{H^4(0,1)} =  2 h^2 \vert u\vert_{H^4(0,1)}.
\end{align*}
Now, consider the function $\psi_1(x):=\tfrac12(x^2-x)$ and observe $\psi_1\in S^*$ and
\begin{equation}\label{eq:bc1}
	0 = (\left(I-\Pi^*\right)u,\psi_1)_{H^2(0,1)} = [\left(I-\Pi^*\right)u]'(1) - [\left(I-\Pi^*\right)u]'(0).
\end{equation}
As $\left(I-\Pi^*\right)u\in H^1_0$, we obtain
\[
	0 = [\left(I-\Pi^*\right)u](1) - [\left(I-\Pi^*\right)u](0) = (\left(I-\Pi^*\right)u,\psi_2'')_{H^1(0,1)},
\]
where $\psi_2(x):=\frac{1}{6} (x^3-x)$.
Integration by parts and $\psi_2''(0)=0$ yields
\[
	0 = (\left(I-\Pi^*\right)u,\psi_2'')_{H^1(0,1)} = - (\left(I-\Pi^*\right)u,\psi_2)_{H^2(0,1)} + [[\left(I-\Pi^*\right)u]' \psi_2''](1).
\]
As $\psi_2\in S^*$, Galerkin orthogonality yields $- (\left(I-\Pi^*\right)u,\psi_2)_{H^2(0,1)} =0$,
so we have
$
	[\left(I-\Pi^*\right)u]'(1) = 0
$.
This implies, in combination with~\eqref{eq:bc1}, that
\[
	u'(1) = (\Pi^*u)' (1)
	\quad \mbox{and} \quad
	u'(0) = (\Pi^*u)' (0)
\]
holds.
So, for any $u\in H^2_0(0,1)$, we have $\Pi^*u\in S^*\cap H^2_0=S^0$. As $\Pi^0$ minimizes the same
norm, we obtain for any $u\in H^2_0(0,1)$ that $\Pi^*u=\Pi^0u$, so also the projector $\Pi^0$
satisfies the desired error estimate.
\end{proof}

\begin{theorem}\label{thrm:appr:c:4}
    Let  $p\in \mathbb{N}$  with $p \ge 3$ and $hp<1$. Then,
    \[
        \| (I-\Pi^0) u \|_{L^2(0,1)} \le 2 h^2 |u|_{H^2(0,1)}\quad \forall\, u\in H^2_0(0,1).
    \]
\end{theorem}
\begin{proof}
This is shown using a classical Aubin Nitsche duality trick. Let $u\in H^2_0(0,1)$ be arbitrary but fixed and choose $v\in H^4(0,1)\cap H^2_0(0,1)$ such that $v'''' = u- \Pi^0 u$.
Using integration by parts and Theorem~\ref{thrm:appr:c:2}, we obtain
\begin{align*}
\| u-\Pi^0u  \|_{L^2(0,1)} &= \frac{\left( u-\Pi^0 u,  u-\Pi^0 u\right)_{L^2(0,1)}}{| u-\Pi^0 u |_{L^2(0,1)}} = \frac{\left( u-\Pi^0 u, v''''\right)_{L^2(0,1)}}{| v |_{H^4(0,1)}}\\
& = \frac{\left( u-\Pi^0 u, v\right)_{H^2(0,1)}}{| v |_{H^4(0,1)}}\leq 2\,h^2 \frac{\left( u-\Pi^0 u, v\right)_{H^2(0,1)}}{| v-\Pi^0v |_{H^2(0,1)}}.
\end{align*}
Galerkin orthogonality gives $\left( u-\Pi^0 u, \Pi^0v\right)_{H^2(0,1)} = 0$. Using this, the Cauchy-Schwarz inequality and this $H^2$-stability of $\Pi^0$, we finally obtain
\begin{align*}
\| u-\Pi^0u  \|_{L^2(0,1)}  &\leq 2\,h^2 \frac{\left( u-\Pi^0 u, v-\Pi^0 v\right)_{H^2(0,1)}}{| v-\Pi^0v |_{H^2(0,1)}}\\ &\leq 2\,h^2 | u-\Pi^0 u|_{H^2(0,1)}\leq 2\, h^2| u |_{H^2(0,1)},
\end{align*}
which finishes the proof.
\end{proof}

\subsection{Approximation error estimates for the parameter domain}

In this subsection, we derive robust approximation error estimates for the space $S^0(\widehat{\Omega})$.
For this purpose, we define the following projectors on $u\in H^{2}(\widehat{\Omega})$:
\begin{align*}
&(\Pi^{x_k})u(x_1,\ldots,x_{k-1},\cdot,x_{k+1},\ldots,x_{d}) := \Pi^0 u(x_1,\ldots,x_{k-1},\cdot,x_{k+1},\ldots,x_{d})\\
&\hspace{2.6cm}\forall\, (x_1,\ldots,x_{k-1},x_{k+1},\ldots,x_{d})\in(0,1)^{d-1}\quad \text{for} \quad k = 1,\ldots,d .
\end{align*}

\begin{lemma}\label{lem:commute}
The projectors $\Pi^{x_k}$ are commutative; that is,
\[
\Pi^{x_i}\Pi^{x_j} = \Pi^{x_j}\Pi^{x_i} \quad \text{for}\quad i= 1,\ldots,d \quad \text{and} \quad j= 1,\ldots,d.
\]
\end{lemma}
\begin{proof}
The proof is completely analogous to that of \cite[Lemma 12]{T:2017MPMG}.
\end{proof}

Let $\widehat{\mathbf{\Pi}}:=\widehat{\mathbf{\Pi}}_{p,h}$ be the $H^2$-orthogonal projection from $H^2_0(\widehat{\Omega})$ into $S^0(\widehat{\Omega})=
S^0_{p,h}(\widehat{\Omega})$.

\begin{theorem}\label{thrm:appr:d:0}
    Let $p\in \mathbb{N}$ with $p\ge 3$ and $hp<1$. Then,
     \begin{equation*}
     |(I-\widehat{\mathbf{\Pi}})u|_{H^2(\widehat{\Omega})} \leq c\,h^2|u|_{H^4(\widehat{\Omega})}\quad \forall\, u\in H^4(\widehat{\Omega}) \cap H^2_0(\widehat{\Omega}).
     \end{equation*}
\end{theorem}
\begin{proof}
For sake of simplicity,
we restrict the proof to the two dimensional case.
Using the triangle inequality and the $H^2$-stability of $\Pi^{x}$, we obtain
\begin{align*}
\|\partial_{xx}(u-\Pi^{x}\Pi^{y} u)\|_{L^2(\widehat{\Omega})} &\leq \|\partial_{xx}(u-\Pi^{x} u)\|_{L^2(\widehat{\Omega})} +\|\partial_{xx}\Pi^{x}(u-\Pi^y u)\|_{L^2(\widehat{\Omega})}\\
&\leq \|\partial_{xx}(u-\Pi^{x} u)\|_{L^2(\widehat{\Omega})} +\|\partial_{xx}(u-\Pi^y u)\|_{L^2(\widehat{\Omega})}.
\end{align*}
Using Theorems~\ref{thrm:appr:c:2} and~\ref{thrm:appr:c:4} and $a+b\le c(a^2+b^2)^{1/2}$, we obtain
\begin{align*}
\|\partial_{xx}(u-\Pi^{x}\Pi^{y} u)\|_{L^2(\widehat{\Omega})}
&\leq c\,h^2\left(\|\partial_{xxxx}u\|^2_{L^2(\widehat{\Omega})} + \|\partial_{xxyy}u\|^2_{L^2(\widehat{\Omega})}\right)^{1/2}.
\end{align*}
Using Lemma~\ref{lem:commute} and the same arguments as above, we obtain
\[
\|\partial_{yy}(u-\Pi^{x}\Pi^{y} u)\|_{L^2(\widehat{\Omega})} \leq c\,h^2\left(\|\partial_{yyxx}u\|^2_{L^2(\widehat{\Omega})} + \|\partial_{yyyy}u\|^2_{L^2(\widehat{\Omega})}\right)^{1/2}.
\]
Using this and Lemma~\ref{lemma:BdecompEquiv2D}, we finally obtain
\begin{align*}
     &|(I-\widehat{\mathbf{\Pi}})u|^2_{H^2(\widehat{\Omega})} 
		\le |\left(I-\Pi^{x}\Pi^{y}\right)u|^2_{H^2(\widehat{\Omega})} 
		= |\left(I-\Pi^{x}\Pi^{y}\right)u|^2_{\mathcal{B}(\widehat{\Omega})}
		\\&\leq 2|\left(I-\Pi^{x}\Pi^{y}\right)u|^2_{\bar{\mathcal{B}}(\widehat{\Omega})}
     = 2\|\partial_{xx}(u-\Pi^{x}\Pi^{y} u)\|^2_{L^2(\widehat{\Omega})} + 2\|\partial_{yy}(u-\Pi^{x}\Pi^{y} u)\|^2_{L^2(\widehat{\Omega})}\\
     &\leq c\,h^4 \left(\|\partial_{xxxx}u\|^2_{L^2(\widehat{\Omega})} + \|\partial_{xxyy}u\|^2_{L^2(\widehat{\Omega})}+ \|\partial_{yyxx}u\|^2_{L^2(\widehat{\Omega})} + \|\partial_{yyyy}u\|^2_{L^2(\widehat{\Omega})} \right) 
		\\&= c\,h^4|u|^2_{H^4(\widehat{\Omega})}.
\end{align*}
\end{proof}

\begin{theorem}\label{thrm:appr:d:1}
    Let $p\in \mathbb{N}$ with $p\ge 3$ and $hp<1$. Then,
     \begin{equation*}
        |(I-\widehat{\mathbf{\Pi}})u|_{H^2(\widehat{\Omega})} \leq c\,h|u|_{H^3(\widehat{\Omega})}\quad \forall\, u\in H^3(\widehat{\Omega}) \cap H^2_0(\widehat{\Omega}).
     \end{equation*}
\end{theorem}
\begin{proof}
    Theorem~\ref{thrm:appr:d:0} states
     \begin{equation*}
        |(I-\widehat{\mathbf{\Pi}})u|_{H^2(\widehat{\Omega})} \leq c\,h^2 |u|_{H^4(\widehat{\Omega})}\quad \forall\, u\in H^4(\widehat{\Omega}) \cap H^2_0(\widehat{\Omega}),
     \end{equation*}
     and, as $\widehat{\mathbf{\Pi}}$ is stable in $H^2$, we have
     \begin{equation*}
        |(I-\widehat{\mathbf{\Pi}})u|_{H^2(\widehat{\Omega})} \leq |u|_{H^2(\widehat{\Omega})}\quad \forall\, u\in H^2_0(\widehat{\Omega}),
     \end{equation*}
     Using~\eqref{eq:h1} for $\theta=1/2$ and~\eqref{eq:sobolev}, we obtain the desired result.
\end{proof}

\begin{theorem}\label{thrm:appr:d:2}
    Let $p\in \mathbb{N}$ with $p\ge 3$ and $hp<1$. Then,
     \begin{equation*}
     |(I-\widehat{\mathbf{\Pi}})u|_{H^1(\widehat{\Omega})} \leq c\,h|u|_{H^2(\widehat{\Omega})}\quad \forall\, u\in H^2_0(\widehat{\Omega}).
     \end{equation*}
\end{theorem}
\begin{proof}
This proof is a variant of the classical Aubin Nitsche duality trick.
Let $u\in H^2_0(\widehat{\Omega})$ be arbitrary but fixed.
Define $f \in H^{-1}(\widehat{\Omega})$ by $\langle f,\cdot\rangle:=(u-\widehat{\mathbf{\Pi}} u,\cdot)_{H^1(\widehat{\Omega})}$
and define $w\in H^2_0(\Omega)$ to be such that 
\[
	(\Delta w,\Delta \widetilde{w})_{L^2(\widehat{\Omega})} = \langle f, \widetilde{w}\rangle \quad \forall\, \widetilde{w} \in H^2_0 (\widehat{\Omega}).
\]
Lax Milgram lemma yields
$|w|_{H^2(\widehat{\Omega})} = \|f\|_{H^{-2}(\widehat{\Omega})}$.
Assumption~\ref{ass:reg} (applied to the parameter domain) implies $w\in H^3(\widehat{\Omega})$
and $|w|_{H^3(\widehat{\Omega})} \le c \|f\|_{H^{-1}(\widehat{\Omega})} = c |u-\widehat{\mathbf{\Pi}} u|_{H^{1}(\widehat{\Omega})}$.
We obtain
\begin{equation*}
|u-\widehat{\mathbf{\Pi}} u|_{H^{1}(\widehat{\Omega})}   =
\frac{(u-\widehat{\mathbf{\Pi}} u,u-\widehat{\mathbf{\Pi}} u)_{H^1(\widehat{\Omega})} }{ |u-\widehat{\mathbf{\Pi}} u|_{H^{1}(\widehat{\Omega})} }	\le c \frac{(u-\widehat{\mathbf{\Pi}} u,u-\widehat{\mathbf{\Pi}} u)_{H^1(\widehat{\Omega})} }{ |w|_{H^{3}(\widehat{\Omega})} }.
\end{equation*}
Using Theorem~\ref{thrm:appr:d:1}, we further obtain
\begin{equation*}
|u-\widehat{\mathbf{\Pi}} u|_{H^{1}(\widehat{\Omega})}   \le c h \frac{(u-\widehat{\mathbf{\Pi}} u,u-\widehat{\mathbf{\Pi}} u)_{H^1(\widehat{\Omega})}  }{ |w-\widehat{\mathbf{\Pi}} w|_{H^{2}(\widehat{\Omega})} }.
\end{equation*}
The definitions of $f$ and $w$, Galerkin orthogonality, Cauchy-Schwarz inequality and the $H^2$-stability of $\widehat{\mathbf{\Pi}}$ yield
\begin{align*}
|u-\widehat{\mathbf{\Pi}} u|_{H^{1}(\widehat{\Omega})}   
		&\le c h \frac{\langle f,u-\widehat{\mathbf{\Pi}} u\rangle  }{ |w-\widehat{\mathbf{\Pi}} w|_{H^{2}(\widehat{\Omega})} }
		= c h \frac{(\Delta w,\Delta( u-\widehat{\mathbf{\Pi}} u) )_{L^2(\widehat{\Omega})}  }{ |w-\widehat{\mathbf{\Pi}} w|_{H^{2}(\widehat{\Omega})} } \\
		&\hspace{-2cm}\le c h \frac{( w, u-\widehat{\mathbf{\Pi}} u )_{H^2(\widehat{\Omega})}  }{ |w-\widehat{\mathbf{\Pi}} w|_{H^{2}(\widehat{\Omega})} }
		= c h \frac{(w-\widehat{\mathbf{\Pi}}w, u-\widehat{\mathbf{\Pi}} u )_{H^2(\widehat{\Omega})}  }{ |w-\widehat{\mathbf{\Pi}} w|_{H^{2}(\widehat{\Omega})} } \le
			ch |u-\widehat{\mathbf{\Pi}} u |_{H^{2}(\widehat{\Omega})} \\
		& \hspace{-2cm} \le	ch |u |_{H^{2}(\widehat{\Omega})},
\end{align*}
which finishes the proof.
\end{proof}

\subsection{Approximation error estimates for the physical domain}

In this subsection, we extend the robust approximation error estimates for the space $S^0(\widehat{\Omega})$
to the space $S^0(\Omega) = S(\Omega)\cap H^2_0(\Omega)$.
For this purpose,
we define $\mathbf{\Pi}=\mathbf{\Pi}_{p,h}$ to be the $H^2$-orthogonal projection from $H^2_0(\Omega) $ into $S^0(\Omega)=S^0_{p,h}(\Omega)$.
Here and in what follows, $h$ always refers to the grid size on the parameter domain. All estimates directly carry over
to the grid size $h_{\Omega}$ on the physical domain because we have $c^{-1} h\le h_{\Omega} \le c h$.

\begin{theorem}\label{thrm:appr:d:3}
     Let $p\in \mathbb{N}$ with $p\ge 3$ and $hp<1$. Then,
     \begin{equation*}
	|\left(I-\mathbf{\Pi}\right)u|_{H^2(\Omega)} \leq c\,h|u|_{H^3(\Omega)}\quad \forall\, u\in H^3(\Omega) \cap H^2_0(\Omega).
     \end{equation*}
\end{theorem}
\begin{proof}
Let $u\in H^3(\Omega) \cap H^2_0(\Omega)$ and $\widehat{u}:=u\circ \textbf{G}$. 
By combining Friedrichs' inequality, Theorem~\ref{thrm:appr:d:1} and Assumption~\ref{ass:GeoEqui}, we obtain
\begin{align*}
\underline{\alpha}\, |u-[\widehat{\mathbf{\Pi}}\widehat{u}]\circ \textbf{G}^{-1}|_{H^2(\Omega)} 
        &\leq \|(I-\widehat{\mathbf{\Pi}})\widehat{u}\|_{H^2(\widehat{\Omega})} 
         \leq c |(I-\widehat{\mathbf{\Pi}})\widehat{u}|_{H^2(\widehat{\Omega})} 
         \leq c\,h|\widehat{u}|_{H^3(\widehat{\Omega})} \\
        &\leq \overline{\alpha} \,c  \,h \|u\|_{H^3(\Omega)}
         \leq \overline{\alpha} \,c  \,h |u|_{H^3(\Omega)},
\end{align*}
where $[\widehat{\mathbf{\Pi}}\widehat{u}]\circ \textbf{G}^{-1} \in S^0( \Omega)$. As $\mathbf{\Pi}$ minimizes the $H^2$-seminorm, we obtain
$|\left(I-\mathbf{\Pi}\right)u|_{H^2(\Omega)}\leq |u-[\widehat{\mathbf{\Pi}}\widehat{u}]\circ \textbf{G}^{-1}|_{H^2(\Omega)}$. Using
	$\overline{\alpha}/\underline{\alpha}\le c$, the desired result follows.
\end{proof}
\begin{theorem}
\label{thrm:appr:d:4}
     Let $p\in \mathbb{N}$ with $p\ge 3$ and $hp<1$ and
     assume that $\Omega$ is such that Assumption~\ref{ass:reg} holds. Then, 
     \begin{equation*}
|\left(I-\mathbf{\Pi}\right)u|_{H^1(\Omega)} \leq c\,h|u|_{H^2(\Omega)}\quad \forall\, u\in H^2_0(\Omega).
     \end{equation*}
\end{theorem}
\begin{proof}
The proof is analogous to the proof of Theorem~\ref{thrm:appr:d:2}. In the proof, we 
use Theorem~\ref{thrm:appr:d:3} instead of Theorem~\ref{thrm:appr:d:1}.
\end{proof}
\section{Stable splitting of the spline space}
\label{sec:splitting}

In this section, we introduce an $L^2$-orthogonal
splitting of the spline space $S^0$ and show that the splitting is stable in $H^2$ analogously to~\cite{hofreither2016robust}.
To do this, we need some more approximation error estimates and inverse inequalities.

\subsection{Approximation error estimates and inverse inequalities}

First, we give an estimate for the periodic case.

\begin{theorem}
\label{thrm:approx1d}
	Let $p\in \mathbb{N}$  with $p \ge 3$ and $hp<1$. Then,
	\[
    		\| (I-T_{p,h}^{2,per}) u \|_{L^2(-1,1)} \le 2 h^2 | u |_{H^2(-1,1)} \quad \forall\, u\in H^2_{per}(-1,1).
    \]
\end{theorem}
\begin{proof}
	Theorem~\ref{thrm:approx1d:prelim} for $q=2$ and $q=3$ can be combined to
	\begin{equation}\nonumber
		|  (I - T_{p,h}^{2,per})u|_{H^{2}(-1,1)} =
		| (I-T_{p,h}^{2,per})(I-T_{p,h}^{3,per}) u |_{H^{2}(-1,1)}  \le 2 h^2 | u |_{H^4(-1,1)} 
	\end{equation}
	for all $u\in H^4_{per}(-1,1)$.
	The desired estimate is shown by an Aubin Nitsche duality trick, which is completely
	analogous to Theorem~\ref{thrm:appr:c:4}.
\end{proof}

Now, we extend the approximation error estimate to non-periodic splines. 
\begin{theorem}
\label{theorem:appDT}
Let  $p\in \mathbb{N}$ with $p \ge 3$ and $hp<1$. Then,
\begin{equation*}
    \| (I-\Pi^{D,0}) u \|_{L^2(0,1)} \le 2 h^2 | u |_{H^2(0,1)} \quad \forall u\in H^2_D(0,1).
\end{equation*}
\end{theorem}
\begin{proof}
    Assume $u\in H^2_D(0,1)$ to be arbitrary but fixed. Define $w$ on $(-1,1)$ by
    \[
            w(x) := \mbox{sign}(x)\; u( |x| )
    \]
    and observe that $w \in H^2_{per}(-1,1)$. Using Theorem~\ref{thrm:approx1d}, we obtain
    \[
        \| (I-T_{p,h}^{2,per}) w \|_{L^2(-1,1)} \le 2 h^2 | w |_{H^2(-1,1)}.
    \]

    First, observe that $|w|_{H^2(-1,1)} = \sqrt{2} |u|_{H^2(0,1)}$.
    Define $w_h:= T_{p,h}^{2,per} w$ and $u_h$ as the restriction of $w_h$. 
    Observe that we obtain $ w_h(x)=-w_h(-x)$ using a standard symmetry argument.
    This implies $u_h\in S^{D,0}$. It follows that ${\|w-w_h\|_{L^2(-1,1)} = \sqrt{2} \|u-u_h\|_{L^2(0,1)}}$. Using this, we obtain
    \[
        \| u-u_h \|_{L^2(0,1)} \le 2 h^2 | u |_{H^2(0,1)}.
    \]

    It remains to show that $u_h$ coincides with $\Pi^{D,0} u$, i.e., that $u-u_h$ is $H^2$-orthogonal to $S^{D,0}$. By definition,
    this means that we have to show
    \begin{equation}\label{eq:proof:theorem:appDT}
        (u-u_h,v_h )_{H^2(0,1)} = 0 \quad \forall\, v_h \in S^{D,0}.
    \end{equation}
    Let  $\widetilde{w}_h \in S^{per}$ be defined as $\widetilde{w}_h := \text{sign}(x) \,v_h( |x| )$ and observe that $(w-w_h,\widetilde{w}_h)_{H^2(-1,1)}= 2(u-u_h,v_h)_{H^2(0,1)}$ since $u$, $u_h$,
    $v_h$ are restrictions of $w$, $w_h$, $\widetilde{w}_h$, respectively. Furthermore, $(w-w_h,\widetilde{w}_h)_{H^2(-1,1)} =0 $
	by construction since $w_h:= T_{p,h}^{2,per} w$. This shows~\eqref{eq:proof:theorem:appDT} and finishes the proof.
\end{proof}

Next, we need an inverse inequality. We extend the $H^1-L^2$-inverse inequality from~\cite{takacs2016approximation} 
to the pair $H^2-L^2$ and the space $S^{D,0}$

\begin{theorem}
\label{theo:BiInv}
For all grid sizes $h$ and each $p\in \mathbb{N}$,
\begin{equation}
\label{eq:inEqBi1}
\vert u_h\vert_{H^{2}\left(0,1\right)} \leq 12h^{-2}\left\Vert u_h\right\Vert_{L^2\left(0,1\right)} \qquad \forall\,u_h\in S^{D,0}.
\end{equation}
\end{theorem}
\begin{proof}
We extend $u_h$ to $(-1,1)$ by defining $w_h(x) = \mbox{sign}(x)\; u_h(\vert x\vert)$. Observe that $w_h\in H^{2,per}(-1,1)$. 
Analogously to the proof of \cite[Theorem~6.1]{takacs2016approximation}, we obtain
\[
\vert w_h' \vert_{H^1(-1,1)} \leq 2\sqrt{3} h^{-1} \left\Vert w_h'\right\Vert_{L^2(-1,1)}
\quad\mbox{and}\quad
\vert w_h \vert_{H^1(-1,1)} \leq 2\sqrt{3} h^{-1} \left\Vert w_h\right\Vert_{L^2(-1,1)}.
\]
The combination of these two results yields
\[
\vert w_h \vert_{H^2(-1,1)} \leq 12h^{-2}\left\Vert w_h\right\Vert_{L^2(-1,1)}.
\]
As $\vert w_h \vert_{H^2(-1,1)} = \sqrt{2} \vert u_h \vert_{H^2(0,1)}$ and
$\left\Vert w_h\right\Vert_{L^2(-1,1)}=\sqrt{2} \left\Vert u_h\right\Vert_{L^2(0,1)}$, the desired result immediately follows.
\end{proof}

\subsection{Stable splitting in the univariate case}

In the previous section, we have introduced the projectors $\Pi^{D,0}: H^{2}_{D}\rightarrow S^{D,0}$. 
Now, we introduce the $L^{2}$-orthogonal projectors
\begin{align*}
Q^{D,0}: S \rightarrow S^{D,0} \quad \mbox{and}\quad
Q^{D,1} := I - Q^{D,0},
\end{align*}
which split $S$ into the direct sum
\begin{align*}
S &= S^{D,0} \oplus S^{D,1} \quad \longleftrightarrow \quad u = Q^{D,0} u + Q^{D,1}u,
\end{align*}
where $S^{D,1}$ is the $L^{2}$-orthogonal complement of $S^{D,0}$ in $S$.
Because the splitting is $L^2$-orthogonal, we obtain
\begin{align}\label{eq:orth:Q}
\left\Vert u\right\Vert^2_{L^2\left(0,1\right)} =\left\Vert Q^{D,0} u\right\Vert^2_{L^2\left(0,1\right)} + \left\Vert Q^{D,1}u\right\Vert^2_{L^2\left(0,1\right)}
\quad \forall\, u\in S.
\end{align}

We show that the splitting is stable in the $H^2$-norm.
\begin{theorem}
\label{theo:split1DH2Stable}
Let  $p\in \mathbb{N}$ with $p\ge 3$ and $hp<1$. Then,
\begin{align*}
c^{-1}\vert u\vert^2_{H^{2}\left(0,1\right)} \leq \vert Q^{D,0} u\vert^2_{H^{2}\left(0,1\right)}  + \vert Q^{D,1}u\vert^2_{H^{2}\left(0,1\right)}  \leq c \vert u\vert^2_{H^{2}\left(0,1\right)}
\quad \forall\, u\in S.
\end{align*}
\end{theorem}
\begin{proof}
The proof is analogous to \cite[Theorem 4]{hofreither2016robust}. The left inequality follows from Cauchy-Schwarz inequality with $c=2$. For the right inequality, we have
\begin{align*}
\vert Q^{D,0} u\vert_{H^{2}\left(0,1\right)} 
& \leq \vert \Pi^{D,0} u \vert_{H^{2}\left(0,1\right)}+ \vert (\Pi^{D,0}  -Q^{D,0} ) u\vert_{H^{2}\left(0,1\right)} \\
& \leq \vert u\vert_{H^{2}\left(0,1\right)} + c h^{-2}\left\Vert (\Pi^{D,0}-Q^{D,0})u\right\Vert_{L^2\left(0,1\right)},
\end{align*}
using the triangle inequality and the inverse inequality Theorem \ref{theo:BiInv}. Using the triangle inequality and the approximation error estimate Theorem \ref{theorem:appDT}, we get
\begin{align*}
&\vert Q^{D,0} u\vert_{H^{2}\left(0,1\right)} \\
&\quad \leq \vert u\vert_{H^{2}\left(0,1\right)} + c h^{-2}\left(\left\Vert (I-\Pi^{D,0})u\right\Vert_{L^2\left(0,1\right)}+\left\Vert (I-Q^{D,0})u\right\Vert_{L^2\left(0,1\right)}\right)\\
&\quad \leq c\vert u\vert_{H^{2}\left(0,1\right)}.
\end{align*}
Using the inequality above together with 
\begin{align*}
\vert Q^{D,0} u\vert_{H^{2}\left(0,1\right)}^2+ \vert Q^{D,1} u\vert_{H^{2}\left(0,1\right)} ^2 \leq 2 \vert u\vert_{H^{2}\left(0,1\right)} ^2+3 \vert Q^{D,0} u\vert_{H^{2}\left(0,1\right)}^2, 
\end{align*}
completes the proof.
\end{proof}

\subsection{Stable splitting in the multivariate case}
The generalization to two and more dimensions is straight forward. Let $\widehat\Omega = (0,1)^d$ and let $\alpha \in \lbrace 0,1\rbrace^d$ be a multiindices. The space $S(\widehat{\Omega})$ is split into the direct sum of $2^d$ subspaces
\begin{equation*}
S(\widehat{\Omega}) = \bigoplus_{\alpha\in\{0,1\}^d} S^{D,\alpha} (\widehat{\Omega}) \quad \text{where} \quad S^{D,\alpha} (\widehat{\Omega}) = S^{D,\alpha_1}\otimes \ldots \otimes  S^{D,\alpha_d}.
\end{equation*}
The $L^2(\widehat\Omega)$-orthogonal projectors are given by 
\begin{equation*} 
\mathbf{Q}^{D,\alpha} := Q^{D,\alpha_1}\otimes\ldots \otimes Q^{D,\alpha_d} : S(\widehat{\Omega}) \rightarrow S^{D,\alpha}(\widehat{\Omega}).
\end{equation*} 
As in the univariate case, the splitting is stable.
\begin{theorem}
\label{theo:stableH1aD}
Let  $p\in \mathbb{N}$ with $p\ge 3$ and $hp<1$. Then,
\begin{align}
\|u\|_{L^2(\widehat\Omega)}^2 &= \sum_{\alpha\in\{0,1\}^d} \|\mathbf{Q}^{D,\alpha}u\|_{L^2(\widehat\Omega)}^2 \quad \forall\, u\in S(\widehat{\Omega}), \label{eq:dd:splitting:equation} \\
c^{-1}\vert u\vert^2_{\bar{\mathcal{B}}(\widehat{\Omega})} &\leq \sum_{\alpha\in\{0,1\}^d}\vert \mathbf{Q}^{D,\alpha} u\vert^2_{\bar{\mathcal{B}}(\widehat{\Omega})}    \leq c \vert u\vert^2_{{\bar{\mathcal{B}}(\widehat{\Omega})}}\quad \forall\, u\in S(\widehat{\Omega}) \label{eq:dd:splitting:inequality}.
\end{align}
\end{theorem}
\begin{proof}
The equation~\eqref{eq:dd:splitting:equation} follows immediately from the equality in the one dimensional case. The left inequality
in~\eqref{eq:dd:splitting:inequality} follows immediately from the Cauchy-Schwarz inequality. 

It remains to show the right inequality in~\eqref{eq:dd:splitting:inequality}.
Let $\alpha$ and $u$ be arbitrary but fixed. We have
\[
	\vert \mathbf{Q}^{D,\alpha} u\vert^2_{{\bar{\mathcal{B}}(\widehat{\Omega})}}   = 
	 \sum_{k=1}^d \| \partial_{x_k}^2 \mathbf{Q}^{D,\alpha} u\|^2_{L^2(\widehat\Omega)} =
	 \sum_{k=1}^d \| \partial_{x_k}^2 Q^{D,\alpha_1} \otimes \cdots \otimes Q^{D,\alpha_d} u\|^2_{L^2(\widehat\Omega)} .
\]
We obtain
\[
	\| \partial_{x_k}^2 Q^{D,\alpha_1} \otimes \cdots \otimes Q^{D,\alpha_d} u\|^2_{L^2(\widehat\Omega)}
	\le c
	\| \partial_{x_k}^2 u\|^2_{L^2(\widehat\Omega)}
\]
by applying~\eqref{eq:orth:Q} for all $Q^{D,\alpha_l}$ with $l\not=k$ and by applying Theorem~\ref{theo:split1DH2Stable}
for $Q^{D,\alpha_k}$. Combining these two inequalities yields
\[
	\vert \mathbf{Q}^{D,\alpha} u\vert^2_{{\bar{\mathcal{B}}(\widehat{\Omega})}}  \le c \vert u\vert^2_{{\bar{\mathcal{B}}(\widehat{\Omega})}}.
\]
Summing over all multi-indices $\alpha$ yields the desired estimate.
\end{proof}

\section{Constructing a robust multigrid method}
\label{sec:MGconstruction}

In this section, we develop a robust multigrid method for solving the linear system~\eqref{eq:lin:sys}. We assume that
we have constructed a hierarchy of grids by uniform refinement. We obtain $V_H\subset V_h$ for two consecutive grids with grid sizes
$h$ and $H:=2h$. For these spaces, we define $P_h:V_{H}\rightarrow V_h$ to be the canonical embedding. We denote the
its matrix representation with the same symbol, the restriction is realized as its transpose~$P_h'$.

For a given initial iterate $u_h^{(k)}$, we obtain the next iterate $u_h^{(k+1)}$ by applying the following steps.
First, we perform $\nu\in\mathbb{N}$ \emph{smoothing steps}, given by 
\begin{equation*}
u^{\left(k, i\right)}_h:= u^{\left(k, i-1\right)}_h +\tau L_h^{-1} \left(f_h-\mathcal{B}_hu^{\left(k, i-1\right)}_h\right), \quad \text{for }i=1,\ldots, \nu,
\end{equation*}
where $u^{\left(k, 0\right)}_h:=u^{(k)}_h$, $L_h$ represents the chosen smoother and $\tau$ is an appropriately chosen damping parameter.
The choice of $L_h$ and $\tau$ is discussed below. Second, we perform a \emph{coarse-grid correction step}, which is for the two-grid method given by
\begin{equation*}
u^{(k+1)}_h:= u^{\left(k, \nu\right)}_h + P_h\mathcal{B}_H^{-1}P_h' \left(f_h-\mathcal{B}_hu^{\left(k, \nu\right)}_h\right).
\end{equation*}

Given a sequence of spaces, we replace the application of $\mathcal{B}_{H}^{-1}$ by one or two steps of the method
on the next coarser level. This results in the V-cycle or W-cycle multigrid method, respectively. The application of $\mathcal{B}_{H}^{-1}$ is realized by means of a direct solver only on the coarsest grid level.  

In the sequel, we discuss two possibilities for the smoother, the Gauss Seidel smoother and a subspace corrected mass smoother.
While only the latter is robust in the spline degree, the Gauss Seidel smoother is superior for small spline degrees and for cases
where a non-trivial geometry transformation is involved. First, we introduce the framework for the convergence analysis and give common results
for both smoothers.

We show the convergence of the multigrid method based on the splitting of the analysis into approximation property and
smoothing property (cf.~\cite{hackbusch2013multi}). As we do not assume full $H^4$-regularity, we choose to show
convergence in the norm $\|\cdot\|_{\mathcal B_h+h^{-2}\mathcal K_h}$, where $\mathcal K_h$ is the matrix 
obtained by discretizing $(\cdot,\cdot)_{H^1(\Omega)}$. The approximation property~\eqref{eq:app} and the
smoothing property~\eqref{eq:smp} read as follows:
\begin{align}
        \| (\mathcal B_h+h^{-2}\mathcal K_h)^{1/2} (I-P_h\mathcal{B}_H^{-1}P_h' \mathcal B_h) \mathcal B_h^{-1} (\mathcal B_h+h^{-2}\mathcal K_h)^{1/2} \| & \le C_A,  \label{eq:app} \\
        \| (\mathcal B_h+h^{-2}\mathcal K_h)^{-1/2} \mathcal B_h (I-\tau L_h^{-1} \mathcal B_h)^\nu (\mathcal B_h+h^{-2}\mathcal K_h)^{-1/2} \| & \le \nu^{-1/2} C_S . \label{eq:smp}
\end{align}
The combination of these two properties yields
\[
        q:=\|(I-P_h\mathcal{B}_H^{-1}P_h' \mathcal B_h) (I-\tau L_h^{-1} \mathcal B_h)^\nu \|_{\mathcal B_h+h^{-2} \mathcal K_h}  \le  \frac{ C_A C_S }{\sqrt{\nu}},
\]
i.e., the two-grid method convergences if sufficiently many smoothing steps are applied. The convergence of the W-cycle multigrid method
follows under weak assumptions (cf.~\cite{hackbusch2013multi}).

The approximation property follows from the approximation error estimates we have shown in Section~\ref{sec:approx:inverse}.
\begin{theorem}\label{thrm:app}
	Let $p\in\mathbb N$ with $p\ge 3$ and $Hp<1$. Then,
    the approximation property~\eqref{eq:app} is satisfied with a constant $C_A$ being independent of $h$ and $p$ (cf. Notation~\ref{not:2:1}).
\end{theorem}
\begin{proof}
    Theorem~\ref{thrm:appr:d:4} states that the $H^2$-orthogonal projector
    $\mathbf{\Pi}_{p,H}^0 : H^2_0(\Omega) \rightarrow V_H = S^0_{p,H}(\Omega)$ satisfies the approximation error estimate 
    \begin{equation*}
            |\left(I-\mathbf{\Pi}_{p,H}^0\right)u|_{H^1(\Omega)} \leq c\,H|u|_{H^2(\Omega)}\quad \forall\, u\in H^2_0(\Omega).
    \end{equation*}
    As the considered functions are in $H^2_0(\Omega)$, Lemma~\ref{lemma:BdecompEquiv2D} implies the same for
    the $\mathcal B$-orthogonal projector.
    For $u_h \in V_h = S^0_{p,h}(\Omega)$, we can rewrite this is matrix-vector notation:    
    \begin{equation*}
            \left\Vert \left(I-P_h\mathcal{B}_H^{-1}P_h'\mathcal{B}_h\right)u_h\right\Vert_{\mathcal{K}_h} \leq c\,H \left\Vert u_h\right\Vert_{\mathcal{B}_h}.
    \end{equation*}
    Using the stability of projectors, we also obtain
    \begin{equation*}
            \left\Vert \left(I-P_h\mathcal{B}_H^{-1}P_h'\mathcal{B}_h\right)u_h\right\Vert_{\mathcal{B}_h} \leq \left\Vert u_h\right\Vert_{\mathcal{B}_h}.
    \end{equation*}
    By combining these two results, we obtain using $H=2h\le ch$ that
    \begin{equation*}
            \left\Vert \left(I-P_h\mathcal{B}_H^{-1}P_h'\mathcal{B}_h\right)u_h\right\Vert_{\mathcal{B}_h+h^{-2} \mathcal{K}_h} \leq c \left\Vert u_h\right\Vert_{\mathcal{B}_h}.
    \end{equation*}
    This reads in matrix-notation as
    \begin{equation*}
            \left\Vert (\mathcal{B}_h+h^{-2} \mathcal{K}_h)^{1/2} \left(I-P_h\mathcal{B}_H^{-1}P_h'\mathcal{B}_h\right)\mathcal{B}_h^{-1/2}\right\Vert\leq c.
    \end{equation*}
    As $\|TT'\|\le\|T\|^2$, we obtain that
    \begin{equation*}
            \left\Vert (\mathcal{B}_h+h^{-2} \mathcal{K}_h)^{1/2} \left(I-P_h\mathcal{B}_H^{-1}P_h'\mathcal{B}_h\right)\left(I-P_h\mathcal{B}_H^{-1}P_h'\mathcal{B}_h\right) \mathcal{B}_h^{-1} (\mathcal{B}_h+h^{-2} \mathcal{K}_h)^{1/2} \right\Vert
    \end{equation*}
    is bounded by some constant $c$
    and, as we have $(I-Q)(I-Q)=I-Q$ for any projector $Q$, the desired statement~\eqref{eq:app}.
\end{proof}

In the two subsequent subsections we show the smoothing estimate
\begin{align}
        \| (\mathcal B_h+h^{-4}\mathcal M_h)^{-1/2} \mathcal B_h (I-\tau L_h^{-1} \mathcal B_h)^\nu (\mathcal B_h+h^{-4}\mathcal M_h)^{-1/2} \| & \le \nu^{-1} \widetilde C_S  \label{eq:smp2}
\end{align}
and the stability estimate
\begin{align}
        \| \mathcal B_h^{1/2} (I-\tau L_h^{-1} \mathcal B_h)^\nu \mathcal B_h^{-1/2} \| & \le 1 . \label{eq:stp2}
\end{align}

Estimate~\eqref{eq:smp2}, together with an $L^2-H^2$-approximation error estimate for the $\mathcal B$-orthogonal projector
would allow to prove a convergence result in the norm $\|\cdot\|_{\mathcal B_h+h^{-4}\mathcal M_h}$, where $\mathcal M_h$ denotes the mass matrix. However,
the proof of such an error estimate requires a full $H^4$-regularity assumption, which is not satisfied in
the cases of interest.

Using Hilbert space interpolation, we obtain the following lemma.
\begin{lemma}\label{lem:interpol}
    The combination of~\eqref{eq:smp2} and~\eqref{eq:stp2} yields~\eqref{eq:smp}, where $C_S$ only depends on $\widetilde C_S$.
\end{lemma}
\begin{proof}
        First observe that Lemma~\ref{lemma:BdecompEquiv2D}, \eqref{eq:smp2} and~\eqref{eq:stp2} yield
        \begin{align*}
           & \| \mathcal B_h(I-\tau L_h^{-1} \mathcal B_h)^\nu u \|_{ [H^2(\Omega)\cap h^{-4}L^2(\Omega)]' }  \le \nu^{-1} \widetilde C_S  \|u\|_{ H^2(\Omega)\cap h^{-4}L^2(\Omega) },\\
           & \| \mathcal B_h(I-\tau L_h^{-1} \mathcal B_h)^\nu u \|_{ [H^2(\Omega)]' } \le |u|_{ H^2(\Omega) }
			\qquad \forall\, u\in V_h,
        \end{align*}
	where $\mathcal B_h$ and $L_h: V_h\rightarrow V_h'$ denote the operator interpretations of the corresponding matrices.
        Using~\eqref{eq:h1} for $\theta=1/2$, \eqref{eq:hilbdual}, \eqref{eq:hilbsum}, \eqref{eq:scaling} and~\eqref{eq:sobolev}, we obtain
        \begin{align*}
            | \mathcal B_h(I-\tau L_h^{-1} \mathcal B_h)^\nu u |_{ [H^2(\Omega)\cap h^{-2}H^1(\Omega)]' }  \le c \widetilde C_S^{1/2} \nu^{-1/2}   |u|_{ H^2(\Omega)\cap h^{-2}H^1(\Omega) },
        \end{align*}
	where $C_S:= c \widetilde C_S^{1/2}$ only depends on $\widetilde C_S$. This directly implies~\eqref{eq:smp}.
\end{proof}

\subsection{Gauss-Seidel smoother}

The most obvious choice of a multigrid smoother is the (symmetric) Gauss-Seidel method.
For simplicity, we restrict ourselves to the symmetric Gauss-Seidel smoother, consisting of one forward Gauss-Seidel
sweep and one backward Gauss-Seidel sweep. Let $\mathcal{B}_h$ be composed into $\mathcal{B}_h = D_h - C_h - C_h'$,
where $C_h$ is a  (strict) left-lower triangular matrix and $D_h$ is a diagonal matrix.
Then, the symmetric Gauss-Seidel method is represented by
\begin{align*}
        L_h := (D_h-C_h)D_h^{-1}(D_h-C_h') = \mathcal{B}_h+C_hD_h^{-1}C_h',
\end{align*}
see, e.g.,~\cite[Note~6.2.26]{hackbusch2013multi}.
Using standard arguments, we can show as follows.
\begin{lemma}
    The matrix $L_h$ satisfies
    \begin{equation}\label{eq:pinch0}
        \mathcal B_h \le L_h \le \mathcal{B}_h+ c(p) h^{-4} \mathcal{M}_h,
    \end{equation}
    where $c(p)$ is independent of the grid size $h$, but depends on the spline degree $p$ and the geometry transformation $\textbf{G}$.
\end{lemma}
\begin{proof}
    As $C_hD_h^{-1}C_h'\ge0$, the first part of the inequality is obvious.

    Now, observe that $\mathcal B_h$ has not more than $\mathcal{O}(p^d)$ non-zero entries per row, so also
    the matrix $D_h^{-1/2} C_h D_h^{-1/2}$ has not more than $\mathcal{O}(p^d)$ non-zero entries per row.
    The absolute value of each of them is bounded by $1$ due
    to the Cauchy-Schwarz inequality. So, we obtain using Gerschorin's theorem that the eigenvalues
    of $D_h^{-1/2} C_h D_h^{-1/2}$ are bounded by $cp^d$, which implies
    \[
        L_h \le \mathcal{B}_h+ cp^{2d} D_h.
    \]
    A standard inverse estimate (cf. \cite[Theorem 3.91]{schwab:1998}) yields
    \[
        L_h \le \mathcal{B}_h+ cp^{2d+8} h^{-4} \mbox{diag } (\mathcal{M}_h) 
    \]
    where $\mbox{diag } (\mathcal{M}_h)$ is the diagonal of the mass matrix $\mathcal{M}_h$. Note that the condition number of the
    B-splines of degree $p$ is bounded by $p2^p$ (cf.~\cite{Scherer:1999}), so we obtain
    \[
        L_h \le \mathcal{B}_h+ c 2^p \,p^{2d+9}\, h^{-4} \mathcal{M}_h,
    \]
	which finishes the proof.
\end{proof}

Now, we can show the convergence of the multigrid method.
\begin{theorem}
\label{theo:GSworks}
    Let $p\in\mathbb{N}$ with $p\ge 3$ and $Hp<1$. Then, there exists a constant $c(p)$, which is indepedent
    of $h$ but depends on $p$ and $\mathbf{G}$,
    such that the two-grid method with the symmetric Gauss-Seidel smoother (with $\tau=1$)
    satisfies
    \[
                q \le \frac{c(p)}{\nu^{1/2}},
    \]
    i.e., it converges if sufficiently many smoothing steps~$\nu$ are applied.
\end{theorem}
\begin{proof}
    From~\eqref{eq:pinch0}, we obtain $\tau L_h^{-1} \mathcal{B}_h \le I$ for $\tau=1$.
    \cite[Lemma~2]{HTZ:2016} implies 
    \[
            \| L_h^{-1/2} \mathcal{B}_h (I-\tau L_h^{-1} \mathcal{B}_h)^{\nu} L_h^{-1/2} \| \le c \nu^{-1},
    \]
    from which the smoothing statement~\eqref{eq:smp2} follows using~\eqref{eq:pinch0}. The stability statement~\eqref{eq:stp2}
    can be shown analogously. Lemma~\ref{lem:interpol} yields the smoothing property~\eqref{eq:smp} with $C_S=c(p)\nu^{-1/2}$.
    
    Theorem~\ref{thrm:app} yields the approximation property~\eqref{eq:app} with $C_A=c$. The combination of smoothing property and
    approximation property yields convergence.
\end{proof}

\subsection{Subspace corrected mass smoother}
\label{ssec:subspace}
We now construct a smoother that satisfies
\begin{equation}\label{eq:pinch}
       c^{-1} \mathcal B_h \le L_h \le c (\mathcal B_h + h^{-4} \mathcal M_h),
\end{equation}
where the constant $c$ is independent of $p$ and $h$ (Notation~\ref{not:2:1}).
To reduce the complexity of the smoother, we construct the local smoothers not around the
original stiffness matrix $\mathcal{B}_h$, representing $(\cdot,\cdot)_{\mathcal{B}(\Omega)}$, but around
the spectrally equivalent matrix $\bar{\mathcal{B}}_h$, representing $(\cdot,\cdot)_{\bar{\mathcal{B}}(\widehat{\Omega})}$. 
Moreover, we observe that the original mass matrix $\mathcal M_h$ is spectrally equivalent
to $\bar{\mathcal{M}}_h$, representing $(\cdot,\cdot)_{L^2(\widehat{\Omega})}$.
Using the spectral equivalence, we obtain that the condition 
\begin{equation}\label{eq:pinch1}
       c^{-1}  \bar{\mathcal{B}}_h \le L_h \le c ( \bar{\mathcal{B}}_h + h^{-4} \bar{\mathcal{M}}_h),
\end{equation}
is equivalent to~\eqref{eq:pinch}.

We follow the ideas of the paper~\cite{hofreither2016robust} and construct local smoothers $L_\alpha$
for any of the spaces $V_{h,\alpha}:= S^{D,\alpha}\cap S^0$, where $\alpha=(\alpha_1,\ldots,\alpha_d)\in\{0,1\}^d$ is a multi-index.
These local contributions are chosen such that they satisfy the corresponding local condition
\begin{equation}\label{eq:pinch2}
       c^{-1} \bar{\mathcal{B}}_\alpha \le L_\alpha \le c (\bar{\mathcal{B}}_\alpha + h^{-4} \bar{\mathcal M}_\alpha),
\end{equation}
where
\[
            \bar{\mathcal{B}}_\alpha := \mathbf{Q}_{h,\alpha} \bar{\mathcal{B}}_h (\mathbf{Q}_{h,\alpha})' \qquad\mbox{and}\qquad
            \bar{\mathcal M}_\alpha := \mathbf{Q}_{h,\alpha} \bar{\mathcal M}_h (\mathbf{Q}_{h,\alpha})'
\]
and $\mathbf{Q}_{h,\alpha}$ is the matrix representation of the canonical embedding $V_{h,\alpha}\rightarrow V_h$.
The canonical embedding has tensor product structure, i.e., $Q_{h,\alpha_1}\otimes\cdots\otimes Q_{h,\alpha_d}$,
where the $Q_{h,\alpha_i}$ are the matrix representations of the corresponding univariate embeddings.

In the two-dimensional case, $\bar{\mathcal{B}}_h$ and $\bar{\mathcal M}_h$ have the representation
\begin{equation*}
\bar{\mathcal{B}}_h = B \otimes M + M \otimes B
\quad\mbox{and}\quad
\bar{\mathcal{M}}_h = M \otimes M,
\end{equation*}
where $M$ and $B$ are the corresponding univariate mass and stiffness matrices.
Restricting $\bar{\mathcal{B}}_h$ to the subspace $V_h^{(\alpha_1,\alpha_2)}$ gives
\begin{equation*}
\bar{\mathcal{B}}_{\alpha_1,\alpha_2} = B_{\alpha_1} \otimes M_{\alpha_2} + M_{\alpha_1} \otimes B_{\alpha_2},
\end{equation*}
where $B_{\alpha_i} = Q_{h,\alpha_i} B (Q_{h,\alpha_i})'$ and $M_{\alpha_i} = Q_{h,\alpha_i} M (Q_{h,\alpha_i})'$.

The inverse inequality for $S^{D,0}$ (Theorem \ref{theo:BiInv}), allows us to estimate
\begin{equation*}
B_0\leq \sigma M_0,
\end{equation*}
where $\sigma = 144 h^{-4}$. Using this, we define the smoothers $L_{\alpha_1,\alpha_2}$ as follows and
obtain estimates for them as follows:
\begin{equation}\label{eq:28a}
\begin{aligned}
\bar{\mathcal{B}}_{00} &\leq 2\sigma M_0 \otimes M_0 &=: L_{00} \le c(\bar{\mathcal{B}}_{00}+h^{-4} \mathcal{M}_{00}),\\
\bar{\mathcal{B}}_{01} &\leq   M_0 \otimes\left(\sigma M_1 +B_1\right) &=: L_{01} \le c(\bar{\mathcal{B}}_{01}+h^{-4} \mathcal{M}_{01}),\\
\bar{\mathcal{B}}_{10} &\leq  \left(B_1  + \sigma M_1 \right)\otimes M_0 &=: L_{10}\le c(\bar{\mathcal{B}}_{10}+h^{-4} \mathcal{M}_{10}),\\
\bar{\mathcal{B}}_{11} &\leq  B_1 \otimes M_1  + M_1 \otimes B_1 &=: L_{11}\le c(\bar{\mathcal{B}}_{11}+h^{-4} \mathcal{M}_{11}).
\end{aligned}
\end{equation}
The extension to three and more dimensions is completely straight-forward (cf.~\cite{hofreither2016robust}).
For each of the subspaces $V_{h,\alpha}$, we have defined a symmetric and positive
definite smoother $L_\alpha$. The overall smoother is given by 
\begin{equation*}
L_h:=\sum_{\alpha\in\{0,1\}^d} (\mathbf{Q}^{D,\alpha})' L_{\alpha} \mathbf{Q}^{D,\alpha},
\end{equation*}
where $\mathbf{Q}^{D,\alpha}= \bar{\mathcal{M}}_{\alpha}^{-1}(\mathbf{Q}_{h,\alpha})'\bar{\mathcal{M}}_{h}$ is the matrix representation of the $L^2$ 
projection from $V_h$ to $V_{h,\alpha}$. Completely
analogous to~\cite[Section~5.2]{hofreither2016robust}, we obtain
\begin{equation*}
L^{-1}_h=\sum_{\alpha\in\{0,1\}^d} \mathbf{Q}_{h,\alpha} L^{-1}_{\alpha} (\mathbf{Q}_{h,\alpha})'.
\end{equation*}

\begin{remark}
How to realize the smoother computationally efficient, is discussed in \cite[Section 5]{hofreither2016robust}.
\end{remark}
The local estimates from~\eqref{eq:pinch2} can be carried over to the whole smoother $L_h$
analogous to the results from \cite{hofreither2016robust}.
\begin{theorem}
\label{theo:alphaCondInv}
Let $p\in \mathbb{N}$ with $p\ge 3$ and $hp<1$.
Assume that
\begin{equation}\label{eq:theo:alphaCondInv}
c^{-1} \bar{\mathcal{B}}_\alpha   \leq L_{\alpha} \leq c(\bar{\mathcal{B}}_\alpha + h^{-4} \bar{\mathcal M}_\alpha )
\qquad\forall\, \alpha \in \lbrace 0,1\rbrace^d.
\end{equation}
Then, the subspace corrected mass smoother satisfies~\eqref{eq:pinch1}.
\end{theorem}
\begin{proof}
Using Theorem \ref{theo:stableH1aD} and~\eqref{eq:theo:alphaCondInv}, we obtain 
\begin{equation*}
\bar{\mathcal{B}}_h  \leq c \sum_{\alpha\in\{0,1\}^d} (\mathbf{Q}^{D,\alpha})'\bar{\mathcal{B}}_h\mathbf{Q}^{D,\alpha} \leq c \sum_{\alpha\in\{0,1\}^d} (\mathbf{Q}^{D,\alpha})'L_h\mathbf{Q}^{D,\alpha}  = c L_h
\end{equation*}
and
\begin{equation*}
L_h   \leq c\sum_{\alpha\in\{0,1\}^d}  (\mathbf{Q}^{D,\alpha})'(\bar{\mathcal{B}}_h+ h^{-4}\bar{\mathcal M}_h)\mathbf{Q}^{D,\alpha} \leq c (\bar{\mathcal{B}}_h+ h^{-4}\bar{\mathcal M}_h),
\end{equation*}
which finishes the proof.
\end{proof}

Now, we can show the robust convergence of the multigrid method.
\begin{theorem}
\label{theo:massSmotherWorks}
    Let $p\in \mathbb{N}$ with $p\ge 3$ and $Hp<1$. Then, there exist two constants $\tau_0$ and $c$ independent of $h$ and $p$
    (cf. Notation~\ref{not:2:1})
    such that for any $\tau\in(0,\tau_0]$ the two-grid method with the subspace corrected mass smother
    satisfies
    \[
                q \le \frac{c\tau^{-1/2}}{\nu^{1/2}},
    \]
    i.e., it converges if sufficiently many smoothing steps~$\nu$ are applied.
\end{theorem}
\begin{proof}
    \eqref{eq:28a} and Theorem~\ref{theo:alphaCondInv} show~\eqref{eq:pinch1}, the spectral equivalence of $\bar{\mathcal{B}}_h$ and $\mathcal{B}_h$
    then shows~\eqref{eq:pinch}. From that estimate, we obtain
    $L_h \ge c^{-1} \mathcal{B}_h$, which implies that there is some constant $\tau_0$
		such that $\tau L_h^{-1} \mathcal{B}_h \le I$ for all $\tau \in (0, \tau_0]$.
    \cite[Lemma~2]{HTZ:2016} implies
    \[
            \| L_h^{-1/2} \mathcal{B}_h (I-\tau L_h^{-1} \mathcal{B}_h)^{\nu} L_h^{-1/2} \| \le c \tau^{-1} \nu^{-1},
    \]
    from which the smoothing statement~\eqref{eq:smp2} follows using~\eqref{eq:pinch}. The stability statement~\eqref{eq:stp2}
    can be shown analogously. Lemma~\ref{lem:interpol} yields the smoothing property~\eqref{eq:smp} with $C_S=c\tau^{-1/2}\nu^{-1/2}$.
    
    Theorem~\ref{thrm:app} yields the approximation property~\eqref{eq:app} with $C_A=c$. The combination of smoothing property and
    approximation property yields convergence.
\end{proof}

\begin{remark}\label{rem:second}
The multigrid methods discussed in this paper can be applied also to the second biharmonic problem
\[
\Delta^2 u = f  \text{ in }\Omega \quad\text{ with }\quad u = \Delta u = 0 \text{ on } \Gamma.
\]
\end{remark}
\begin{remark}\label{rem:third}
The multigrid methods discussed in this paper can be applied also to the third biharmonic problem
\[
\Delta^2 u = f  \text{ in }\Omega \quad\text{ with }\quad \nabla u \cdot \mathbf{n}= \nabla \Delta u \cdot \mathbf{n} = 0 \text{ on } \Gamma
\]
on the parameter domain.
In this case, the subspace corrected mass smoother has to be based on the splitting of $S$ into the space of functions in $S$ whose odd derivatives vanish on the boundary and its orthogonal complement. This is the same splitting which was used
in~\cite{hofreither2016robust}. How to transform a strong formulation of the boundary condition to the physical domain, is not obvious.
\end{remark}

\section{Numerical results}\label{sec:num}

In this section, we compare multigrid solvers based on the two smoothers introduced
in Section~\ref{sec:MGconstruction}, the symmetric Gauss-Seidel smoother and the subspace corrected mass smoother. 
This is done first for a problem with a trivial geometry transformation, then for a problem with a nontrivial geometry transformation.

All numerical experiments are implemented using the G+Smo library~\cite{gismoweb}.

\subsection{Experiments on the parameter domain}
We solve the model problem on the unit square and the unit cube; that is,
\begin{equation*}
\Delta^2 u = f \quad \text{in}\quad \Omega:=(0,1)^d \quad \text{with}\quad u = \nabla u\cdot \mathbf{n} = 0\quad \text{on}\quad\Gamma,
\end{equation*}
for $d=2,3$ with the right-hand side
\begin{equation*}
f(x_1,\ldots,x_d) := d^2\pi^4 \prod^d_{j=1}\sin{\left(\pi x_j\right)}.
\end{equation*}
The problem is discretized using tensor product B-splines with equidistant knot spans and maximum continuity.

\newcommand{\xx}{OoM}

\begin{table}[ht]
\begin{center} 
  \begin{tabular}{| c || c | c | c | c | c | c | c | c |}
    \hline
    $\ell$ $\backslash\; p$ & \hspace{.2cm}3\hspace{.2cm} & \hspace{.2cm}4\hspace{.2cm} & \hspace{.2cm}5\hspace{.2cm} & \hspace{.2cm}6\hspace{.2cm} & \hspace{.2cm}7\hspace{.2cm} & \hspace{.2cm}8\hspace{.2cm} & \hspace{.2cm}9\hspace{.2cm} & \hspace{.15cm}10\hspace{.15cm} \\ \hline \hline
\multicolumn{9}{|l|}{Symmetric Gauss-Seidel} \\ \hline 
    5    &  5 &     9&    18&    32&    60&   117& 204 & 389 \\  \hline
    6    &  5 &     9&    18&    33&    59&   115& 215 & 400 \\  \hline
    7    &  5 &     9&    17&    32&    60&   107& 210 & 395 \\  \hline
    8    &  5 &     9&    17&    32&    60&   112& 197 & 375 \\  \hline\hline
\multicolumn{9}{|l|}{Subspace corrected mass smoother } \\ \hline
    5    & 40 &    39&    38&    35&    33&   30& 28 & 26 \\  \hline
    6    & 41 &    41&    41&    40&    38&   37& 35 & 34 \\  \hline
    7    & 41 &    42&    42&    41&    40&   39& 37 & 36 \\  \hline
    8    & 42 &    42&    42&    42&    41&   39& 38 & 37 \\  \hline
  \end{tabular}
  \caption{Iteration counts for the unit square.}
  \label{t:Para2D}
\end{center}
\end{table}

\begin{table}[ht]
\begin{center} 
 \begin{tabular}{| c || c | c | c | c | c |}
   \hline
   $\ell$ $\backslash\; p$ & \hspace{.3cm}3\hspace{.3cm} & \hspace{.3cm}4\hspace{.3cm} & \hspace{.3cm}5\hspace{.3cm} & \hspace{.3cm}6\hspace{.3cm} & \hspace{.3cm}7\hspace{.3cm} \\ \hline \hline
\multicolumn{6}{|l|}{Symmetric Gauss-Seidel} \\ \hline \hline
   3    & 11 &    29&     81&   217&   676 \\  \hline
   4    & 12 &    31&     83&   218&   575 \\  \hline
   5    & 13 &    32&     82&   213&   537 \\  \hline
   6    & 13 &    32&     83&   211&   528 \\  \hline \hline
\multicolumn{6}{|l|}{Subspace corrected mass smoother } \\  \hline 
   3    & 33 &    23&    18&    16&    15 \\  \hline
   4    & 45 &    41&    36&    32&    28 \\  \hline
   5    & 50 &    50&    48&    46&    43 \\  \hline
   6    & 52 &    53&    52&    51&    49 \\  \hline
 \end{tabular}
 \caption{Iteration counts for the unit cube.}
 \label{t:Para3D}
\end{center}
\end{table}

We solve the resulting system using a conjugate gradient (CG) solver,
preconditioned with one multigrid V-cycle with $1$ pre and $1$ post smoothing step.
When using the W-cycle, which is covered by the convergence theory, one obtains
comparable iteration counts; as the V-cycle is more efficient, we present our results for that case.
When using the \emph{Gauss-Seidel smoother}, we perform the multigrid method directly on the system matrix $\mathcal{B}_h$.
When using the \emph{subspace corrected mass smoother}, we perform the multigrid method on the auxiliary operator $\bar{\mathcal{B}}_h$,
representing the reduced inner product $(\cdot,\cdot)_{\bar{\mathcal{B}}(\Omega)}$. Here, we use 
that the matrices $\mathcal{B}_h$ and $\bar{\mathcal{B}}_h$ are spectrally equivalent with constants independent of $p$ and $h$.
For the subspace corrected mass smoother, we choose $\sigma^{-1} := 0.015h^{4}$ for $d=2$ and 
$\sigma^{-1} := 0.020h^{4}$ for $d=3$. In all cases, we choose $\tau:=1$.

The initial guess is a random vector.
Tables~\ref{t:Para2D} and \ref{t:Para3D} show  the number of iterations needed to
reduce the initial residual by a factor of $10^{-8}$ for the unit square and the unit cube. 
We do the experiments for several choices of the spline degree $p$ and several
uniform refinement levels $\ell$. (The refinement level $\ell=0$ corresponds to the
domain consisting only of one element.)
On the finest considered grid, the number of degrees of freedom
ranges for $d=2$ between around $65$ and $69$ thousand and for $d=3$ between
$250$ and $301$ thousand. The number of non-zero entries of the stiffness
matrix ranges for $d=2$ between around $3$ and $29$ million and for $d=3$ between 
$79$ and $855$ million.

As predicted, the iteration counts of the multigrid solver with Gauss-Seidel smoother heavily
depend on $p$. This effect is amplified in the three dimensional case. The mass smoother
(which is proven to be $p$-robust) outperforms the Gauss-Seidel smoother
for $p\ge7$ in the two dimensional case and for $p\ge5$ in the three dimensional case.

\subsection{Experiments on nontrivial computational domains}

In this subsection, we present the results for the same model problem as in the previous subsection, but on the
nontrivial geometries shown in Figures~\ref{fig:domains2d}~and~\ref{fig:domains3d}.
\begin{figure}[h]
  \center
\begin{minipage}{0.47\textwidth}
  \centering \includegraphics[width=0.56\textwidth]{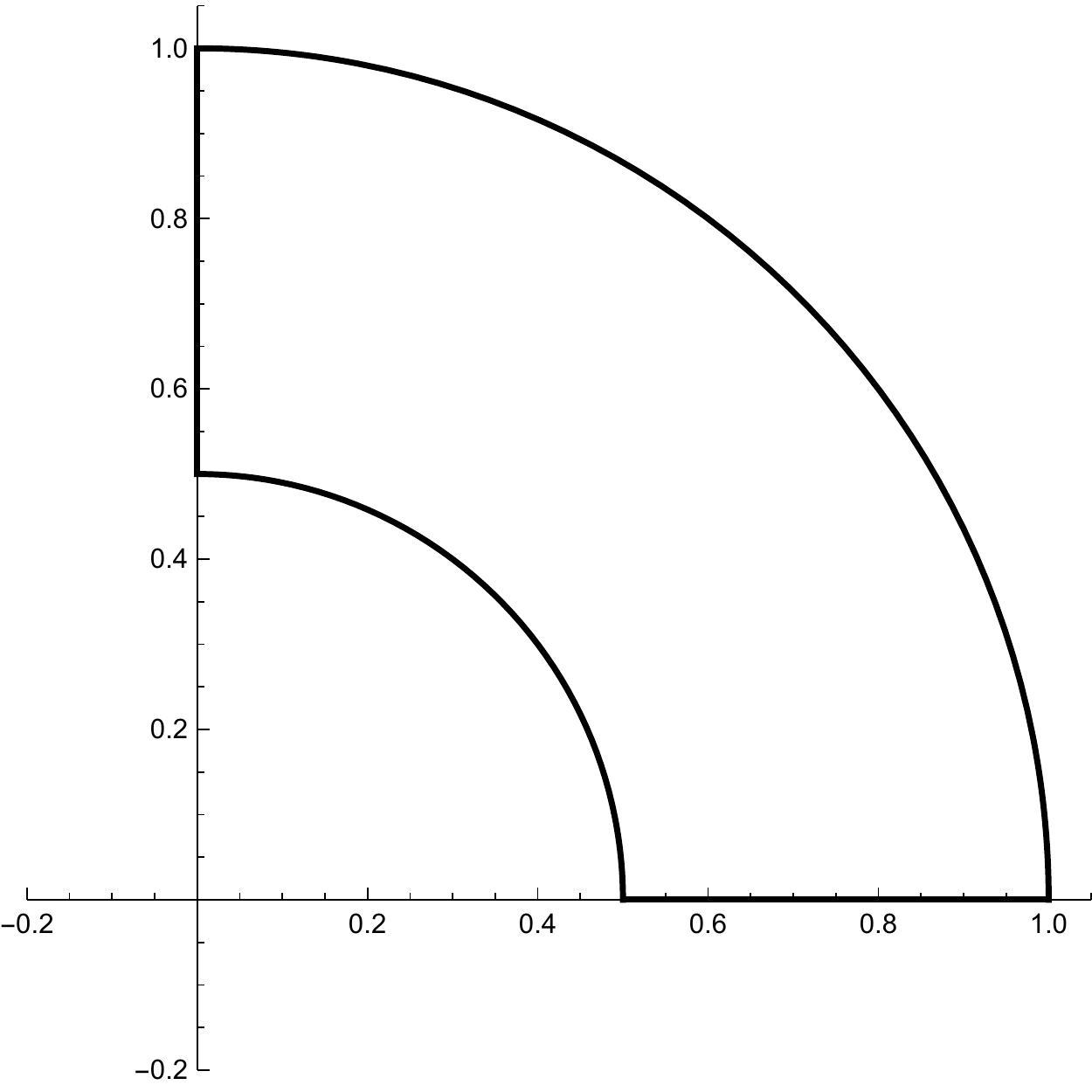}
  \caption{The two-dimensional domain}
  \label{fig:domains2d}
\end{minipage}
\begin{minipage}{0.47\textwidth}
  \centering \includegraphics[width=0.56\textwidth]{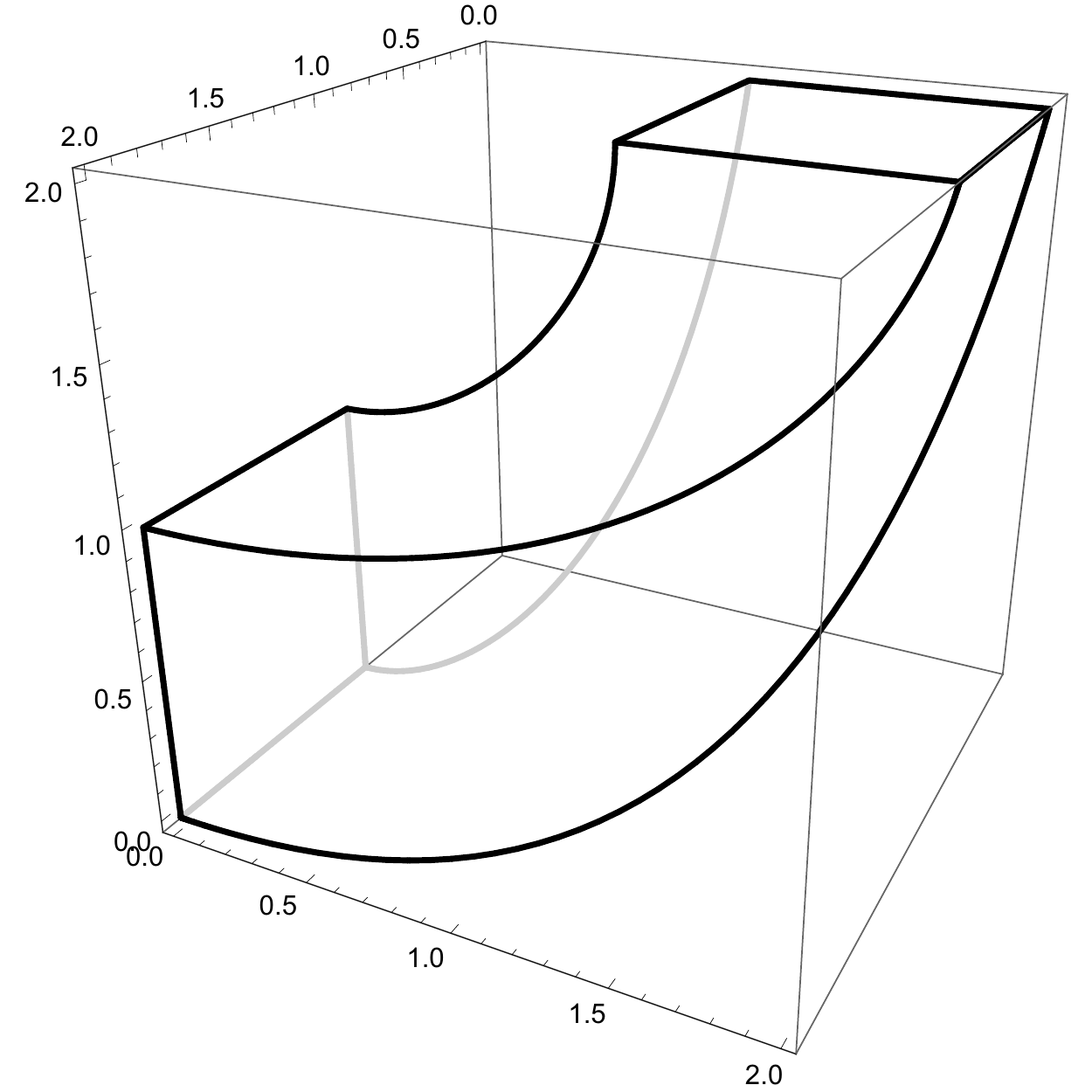}
  \caption{The three-dimensional domain}
  \label{fig:domains3d}
\end{minipage}
\end{figure}

When using the \emph{Gauss-Seidel smoother}, we again perform the multigrid method directly on the system matrix $\mathcal{B}_h$.
When using the \emph{subspace corrected mass smoother}, we perform the multigrid method on the auxiliary operator $\bar{\mathcal{B}}_h$,
representing the reduced inner product $(\cdot,\cdot)_{\bar{\mathcal{B}}(\widehat{\Omega})}$ on the \emph{parameter domain}.
Again, we use that the  matrices $\mathcal{B}_h$ and $\bar{\mathcal{B}}_h$ are spectrally equivalent with constants
independent of $p$ and $h$, but which certainly depend on the geometry transformation.
For the subspace corrected mass smoother, we choose $\sigma^{-1} := 0.015h^{4}$ for $d=2$ and $\sigma:=0.020h^{4}$ for $d=3$. Again,
we choose $\tau=1$ in all cases.

\begin{table}[ht]
\begin{center} 
  \begin{tabular}{| c || c | c | c | c | c | c | c | c | c |}
    \hline
    $\ell$ $\backslash\; p$
    & \hspace{.2cm}3\hspace{.2cm} & \hspace{.2cm}4\hspace{.2cm} & \hspace{.2cm}5\hspace{.2cm} & \hspace{.2cm}6\hspace{.2cm} & \hspace{.2cm}7\hspace{.2cm} & \hspace{.2cm}8\hspace{.2cm} & \hspace{.2cm}9\hspace{.2cm} & \hspace{.15cm}10\hspace{.15cm} \\ \hline \hline
\multicolumn{9}{|l|}{Symmetric Gauss-Seidel} \\ \hline 

    5    & 15 & 15 & 20 & 37 & 69  & 133 & 220 & 413 \\  \hline
    6    & 17 & 16 & 21 & 37 & 66  & 127 & 234 & 428 \\  \hline
    7    & 18 & 17 & 21 & 37 & 68  & 125 & 231 & 413 \\  \hline
    8    & 19 & 17 & 21 & 37 & 67  & 120 & 217 & 380 \\  \hline\hline
\multicolumn{9}{|l|}{Subspace corrected mass smoother } \\ \hline

    5     & 162 & 161 & 152 & 150 & 142 & 134 & 130 & 127 \\  \hline
    6     & 196 & 200 & 200 & 194 & 180 & 179 & 178 & 171 \\  \hline
    7     & 215 & 220 & 225 & 222 & 219 & 210 & 198 & 198 \\  \hline
    8     & 226 & 232 & 243 & 233 & 227 & 221 & 217 & 210 \\  \hline
  \end{tabular}
  \caption{Iteration counts for 2D physical domain given in Figure~\ref{fig:domains2d}}
  \label{t:Ann2D}
\end{center}
\end{table}

\begin{table}[ht]
\begin{center} 
 \begin{tabular}{| c || c | c | c | c | c |}
   \hline
   $\ell$ $\backslash\; p$ & 3 & 4 & 5 & 6 & 7 \\ \hline \hline
 \multicolumn{6}{|l|}{Symmetric Gauss-Seidel} \\ \hline 
   3    & 14 & 32 &  93 & 262 &  763 \\  \hline
   4    & 23 & 35 &  94 & 246 &  634 \\  \hline
   5    & 36 & 37 &  88 & 226 &  516 \\  \hline
   6    & 51 & 45 &  90 & 220 &  \xx \\   \hline\hline
\multicolumn{6}{|l|}{Subspace corrected mass smoother } \\ \hline
   3    & 115 & 114 & 130 & 142 & 154 \\  \hline
   4    & 259 & 243 & 241 & 235 & 233 \\  \hline
   5    & 443 & 441 & 430 & 410 & 380 \\  \hline
   6    & 651 & 650 & 644 & 637 & \xx \\   \hline
 \end{tabular}
 \caption{Iteration counts for 3D physical domain given in Figure~\ref{fig:domains3d}}
 \label{t:Ann3D}
\end{center}
\end{table}

Tables~\ref{t:Ann2D} and \ref{t:Ann3D} show the number of iterations\footnote{
The entry \xx{} indicates that we ran out of memory when assembling the stffness matrix.
} required to reduce the initial residual by a factor of $10^{-8}$.
Again, we obtain very nice results for the Gauss Seidel smoother which -- as for the case of trivial computational domains --
deteriorate if $p$ is increased.

For the mass smoother, we have proven robustness in $p$ and $h$.
Here, the results might look like the mass smoother is not robust in $h$. 
The reason is that
a sufficiency small grid size $h$ is needed to capture the full effect
of the geometry transformation. A similar observation can also be made 
for the Possion problem (cf.~\cite[Table 4]{hofreither2016robust}).
The effects of the geometry transformation can be measured by the condition number of $\bar{\mathcal{B}}_h^{-1} \mathcal{B}_h$.
For the Poisson problem, this condition number was estimated, e.g., in~\cite{sangalli2016isogeometric}. For the biharmonic problem, the condition number is typically the square of the condition number for the Poisson problem, which explains that
the dependence on the geometry transformation is more severe for the biharmonic problem.

\subsection{A hybrid smoother}\label{sec:hybrid}

The numerical experiments have shown that the Gauss-Seidel smoother captures the effects of the geometry transformation quite well
and that it is superior to the mass smoother for nontrivial domains, unless $p$ is particularly high.
The mass smoother is robust in $p$, but does not perform well for non-trivial geometries.
So, it seems to be a good idea to set up a \emph{hybrid smoother} which combines the strengths of both proposed smoothers.

We set up again a conjugate gradient solver, preconditioned with one multigrid V-cycle with $1$ pre and $1$ post smoothing
step. Here, in order to represent the geometry well, the multigrid solver is set up on the original system matrix $\mathcal{B}_h$.
The hybrid smoother consists of one forward Gauss-Seidel sweep, followed by one step of the subspace corrected mass smoother,
finally followed by one backward Gauss-Seidel sweep. As always, the subspace corrected mass smoother -- which requires
a tensor-product matrix -- is constructed based on the reduced matrix $\bar{\mathcal{B}}_h$ on the parameter domain.
For the Gauss-Seidel sweeps, we choose $\tau=1$; and for the subspace corrected mass smoother, we choose 
$\tau = 0.125$ and $\sigma^{-1} = 0.015h^{4}$ for $d=2$ and $\tau = 0.09$ and $\sigma^{-1} = 0.015h^{4}$ for $d=3$.

Tables~\ref{t:Ann2D:hyb} and \ref{t:Ann3D:hyb} show the iteration numbers for the hybrid smoother. We see that the iteration counts
are quite robust both in the spline degree $p$ and in the grid level $\ell$. For small spline degrees, the iteration
counts are comparable to the multigrid preconditioner with Gauss Seidel smoother. For high spline degrees, the hybrid smoother
outperforms both other approaches, even if one considers that the overall costs for one step the hybrid smoother are comparable
to the overall costs of two smoothing steps of one of the other smoothers.

\begin{table}[ht]
\begin{center} 
  \begin{tabular}{| c || c | c | c | c | c | c | c | c | c |}
    \hline
    $\ell$ $\backslash\; p$
    & \hspace{.2cm}3\hspace{.2cm} & \hspace{.2cm}4\hspace{.2cm} & \hspace{.2cm}5\hspace{.2cm} & \hspace{.2cm}6\hspace{.2cm} & \hspace{.2cm}7\hspace{.2cm} & \hspace{.2cm}8\hspace{.2cm} & \hspace{.2cm}9\hspace{.2cm} & \hspace{.15cm}10\hspace{.15cm} \\ \hline \hline
\multicolumn{9}{|l|}{Hybrid smoother } \\ \hline
    5     & 15  & 14  & 16  & 19  & 22  & 24  & 26  & 28 \\  \hline
    6     & 16  & 15  & 17  & 20  & 23  & 26  & 30  & 31 \\  \hline
    7     & 18  & 16  & 18  & 21  & 25  & 28  & 31  & 32 \\  \hline
    8     & 19  & 16  & 19  & 22  & 25  & 28  & 32  & 33 \\  \hline
  \end{tabular}
  \caption{Iteration counts for 2D physical domain given in Figure~\ref{fig:domains2d}}
  \label{t:Ann2D:hyb}
\end{center}
\end{table}

\begin{table}[ht]
\begin{center} 
  \begin{tabular}{| c || c | c | c | c | c |}
    \hline
    $\ell$ $\backslash\; p$ & \hspace{.3cm}3\hspace{.3cm} & \hspace{.3cm}4\hspace{.3cm} & \hspace{.3cm}5\hspace{.3cm} & \hspace{.3cm}6\hspace{.3cm} & \hspace{.3cm}7\hspace{.3cm} \\ \hline \hline
\multicolumn{6}{|l|}{Hybrid smoother } \\ \hline 
        3    & 13 & 14 & 17 & 22 & 26 \\  \hline
        4    & 23 & 22 & 23 & 26 & 31 \\  \hline
        5    & 36 & 34 & 33 & 36 & 41 \\  \hline
        6    & 51 & 44 & 44 & 46 &\xx \\  \hline
  \end{tabular}
  \caption{Iteration counts for 3D physical domain given in Figure~\ref{fig:domains3d}}
  \label{t:Ann3D:hyb}
\end{center}
\end{table}

\section*{Acknowledgments}
The research of the first author was supported
by the Austrian Science Fund (FWF): S11702-N23.

\bibliographystyle{elsarticle-num}
\bibliography{bibliography}

\end{document}